\documentclass[11pt]{amsart}

\issueinfo{00}{0}{Month}{Year}
\copyrightinfo{2008}{University of Calgary}
\pagespan{1}{2}
\PII{ISSN 1715-0868}

\usepackage{amsfonts, amsbsy, amsmath, amssymb, mathtools, amsthm, amsxtra, amscd}

\newtheorem{thm}{Theorem}[section]
\newtheorem{lem}[thm]{Lemma}
\newtheorem{cor}[thm]{Corollary}
\newtheorem{prop}[thm]{Proposition}

\newtheorem{rmk}[thm]{Remark}

\newtheorem{thm-con}[thm]{Theorem-Conjecture}
\numberwithin{equation}{section}

\theoremstyle{definition}

\newcommand{\f}{\Bbb F}

\begin{document}

\title[Reversed Dickson polynomials of the $(k+1)$-th kind over finite fields, II]{Reversed Dickson polynomials of the $(k+1)$-th kind over Finite Fields, II}

\author[Neranga Fernando]{Neranga Fernando}
\address{Department of Mathematics and Computer Science,
College of the Holy Cross, Worcester, MA 01610, USA}
\email{nfernand@holycross.edu}

\date{(date1), and in revised form (date2).}
\subjclass[2010]{11T55, 11T06}
\keywords{Finite fields, Reversed Dickson polynomials, Permutation polynomials}

\begin{abstract}
Let $p$ be an odd prime. In this paper, we study the permutation behaviour of the reversed Dickson polynomials of the $(k+1)$-th kind $D_{n,k}(1,x)$ when $n=p^{l_1}+3$, $n=p^{l_1}+p^{l_2}+p^{l_3}$, and $n=p^{l_1}+p^{l_2}+p^{l_3}+p^{l_4}$, where $l_1, l_2$, $l_3$, and  $l_4$ are non-negative integers. A generalization to $n=p^{l_1}+p^{l_2}+\cdots +p^{l_i}$ is also shown. We find some conditions under which $D_{n,k}(1,x)$ is not a permutation polynomial over finite fields for certain values of $n$ and $k$. We also present a generalization of a recent result regarding $D_{p^l-1,1}(1,x)$ and present some algebraic and arithmetic properties of $D_{n,k}(1,x)$.
\end{abstract}

\maketitle

\tableofcontents


\section{Introduction}

Let $p$ be an odd prime and $q=p^e$, where $e$ is a positive integer. Let $\Bbb F_q$ be the finite field with $q$ elements. A polynomial $f \in \Bbb F_q[{\tt x}]$ is called a \textit{permutation polynomial} (PP) of $\Bbb F_q$ if the associated mapping $x\mapsto f(x)$ from $\f_q$ to $\f_q$ is a permutation of $\Bbb F_q$. Permutation polynomials over finite fields have important applications in coding theory, cryptography, finite geometry, combinatorics, and computer science, among other fields. 

In the study of permutation polynomials, Dickson polynomials have played a major role. The $n$-th Dickson polynomial of the first kind $D_n(x,a)$ is defined by
\[
D_{n}(x,a) = \sum_{i=0}^{\lfloor\frac n2\rfloor}\frac{n}{n-i}\dbinom{n-i}{i}(-a)^{i}x^{n-2i},
\]
where $a\in \f_q$ is a parameter. 

The permutation property of the Dickson polynomials of the first kind is completely known. When $a=0$, $D_{n}(x,a) = x^n$, which is a PP over $\Bbb F_q$ if and only if $(n,q-1)=1$. When $0\neq a \in \mathbb{F}_{q}$, $D_{n}(x,a)$ is a PP over $\Bbb F_q$ if and only if $(n,q^2-1)=1$; see \cite[Theorem ~7.16]{Lidl-Niederreiter-97} or \cite[Theorem ~3.2]{Lidl-Mullen-Turnwald-1993}.

The $n$-th reversed Dickson polynomial of the first kind $D_{n}(a,x)$ was first introduced by Hou, Mullen, Sellers and Yucas in \cite{Hou-Mullen-Sellers-Yucas-FFA-2009} by reversing the roles of the variable and the parameter in the $n$-th Dickson polynomial of the first kind $D_{n}(x,a)$. The $n$-th reversed Dickson polynomial of the first kind $D_n(a,x)$ is defined by

\[
D_{n}(a,x) = \sum_{i=0}^{\lfloor\frac n2\rfloor}\frac{n}{n-i}\dbinom{n-i}{i}(-x)^{i}a^{n-2i},
\]

where $a\in \f_q$ is a parameter. 

In 2012, Wang and Yucas introduced the $n$-th reversed Dickson polynomial of the $(k+1)$-th kind (\cite{Wang-Yucas-FFA-2012}). Let $k$ be an integer such that $0\leq k\leq p-1$. For $a\in \f_q$, the $n$-th reversed Dickson polynomial of the $(k+1)$-th kind $D_{n,k}(a,x)$ is defined by
\begin{equation}\label{E1.1}
D_{n,k}(a,x) = \sum_{i=0}^{\lfloor\frac n2\rfloor}\frac{n-ki}{n-i}\dbinom{n-i}{i}(-x)^{i}a^{n-2i},
\end{equation}

\noindent and $D_{0,k}(a,x)=2-k$.

In \cite{Hong-Qin-Zhao-FFA-2016-2}, Hong, Qin, and Zhao studied reversed Dickson polynomials of the second kind and presented several necessary conditions for the reversed Dickson polynomial of the second kind to be a permutation of $\f_q$. In \cite{Fernando-2016}, the author of the present paper explored reversed Dickson polynomials of the third kind and found the necessary conditions for them to be a permutation on $\f_q$. 

In \cite{Fernando-2016-3}, the author of this paper unified and generalized several recently discovered results on reversed Dickson polynomials over finite fields. It was shown in \cite{Fernando-2016-3} that to discuss the permutation behaviour of reversed Dickson polynomials of the $(k+1)$-th kind, one only has to consider $a=1$. We completely explained the permutation behaviour of the reversed Dickson polynomials of the $(k+1)$-th kind $D_{n,k}(1,x)$ over finite fields when $n=p^l$, $n=p^l+1$, and $n=p^l+2$, where $p$ is an odd prime and $l\geq 0$ is an integer, in \cite{Fernando-2016-3} and \cite{Fernando-2016-4}. In the present paper, we explore the permutation behaviour of $D_{n,k}(1,x)$ over finite fields when $n$ is a sum of odd prime powers. 

The observations in \cite{Fernando-2016-3} and \cite{Fernando-2016-4} led to the question ``when is $D_{n,k}(1,x)$ a PP of $\f_{p^e}$ when $n=p^l+3$?''. The case $n=p^l+3$ led us to consider the cases $n=p^{l_1}+p^{l_2}+p^{l_3}$ and $n=p^{l_1}+p^{l_2}+p^{l_3}+p^{l_4}$ which eventually led to a generalization, $n=p^{l_1}+p^{l_2}+\cdots +p^{l_i}$. As a consequence, we explain the permutation behaviour of $D_{p^{l_1}+p^{l_2},k}(1,x)$.

We would like to point out that the generalization to $n=p^{l_1}+p^{l_2}+\cdots +p^{l_i}$ given in Section~\ref{S6} paves the way to many combinatorial problems. It opens the door to find explicit polynomials that have the same permutation behaviour as $D_{n,k}(1,x)$ when $n=rp^l+s$, where $r\geq 1$ and $l, s\geq 0$ are integers. This allows us to consider two special cases, $n=p^l+s$ and $n=rp^l$. As a consequence, we find some conditions under which $D_{n,k}(1,x)$ does not permute $\f_{p^e}$. We believe that our results would be helpful for anyone who is interested in investigating the permutation behaviour of reversed Dickson polynomials of the $(k+1)$-th kind over finite fields.

The paper is organized as follows. 

In Section 2, we present some preliminaries. We explore the permutation behaviour of $D_{n,k}(1,x)$ when $n=p^l+3$, $n=p^{l_1}+p^{l_2}+p^{l_3}$, and $n=p^{l_1}+p^{l_2}+p^{l_3}+p^{l_4}$ in Sections 3, 4 and 5, respectively. 

In Section 6, we present the generalization to $n=p^{l_1}+p^{l_2}+\cdots +p^{l_i}$. In Section 7, we consider the cases $n=rp^l+s$, where $r\geq 1$ and $l, s\geq 0$ are integers. As a result, we study two special cases, $n=p^l+s$ and $n=rp^l$. 

In Section 8, we consider the case $n=rp^l-1$, which generalizes a previous result on reversed Dickson polynomials of the second kind obtained in \cite{Hong-Qin-Zhao-FFA-2016-2} by Hong, Qin, and Zhao. We also present a generalization of a result when $p=3$ that appeared in \cite{Hou-Mullen-Sellers-Yucas-FFA-2009}. Next, we show that the $n$-th reversed Dickson polynomial of the $(k+1)$-th kind $D_{n,k}(1,x)$ can be written in terms of $D_{n-2,1}(1,x)$ and $D_{n,0}(1,x)$. We also show that the $n$-th reversed Dickson polynomial of the $(k+1)$-th kind $D_{n,k}(1,x)$ can be written in terms of $D_{n-1,2}(1,x)$ and $D_{n,0}(1,x)$. Motivated by a matrix representation of the Dickson polynomials of the first kind that appeared in \cite[Chapter 2]{Lidl-Mullen-Turnwald-1993}, we present a matrix representation of the $n$-th reversed Dickson polynomial of the $(k+1)$-th kind $D_{n,k}(1,x)$. 

Throughout the paper, we denote $D_{n,0}(1,x)$ by $D_{n}(1,x)$ and we always assume that $p$ is odd unless otherwise specified. We would also like to point out that some computations are omitted in this paper due to their length. We refer the interested reader to \cite{Fernando-2017-II} for detailed computation.


\section{Preliminaries}

For $a\neq 0$, we write $x=y(a-y)$ with an indeterminate $y\in \f_{q^2}$ such that $y\neq \frac{a}{2}$.  Then we have

\begin{equation}\label{E2.1}
\begin{split}
D_{n,k}(a,x)&=k\,\Big[ \displaystyle\frac{y^n(a-y)-y(a-y)^n}{2y-a}\Big]+D_n(a,x),
\end{split}
\end{equation}

\noindent where $D_n(a,x)$ denotes the $n$-th reversed Dickson polynomial of the first kind; see \cite[Eq. 2.1]{Fernando-2016-3}. 

Let  $a\in \f_{q}^{*}$. Then it follows from the definition that
\[
D_{n,k}(a,x)=a^n\,D_{n,k}(1,\frac{x}{a^2}). 
\]
Hence $D_{n,k}(a,x)$ is a PP on $\f_{q}$ if and only if $D_{n,k}(1,x)$ is a PP on $\f_{q}$. Hereafter we always assume $a=1$.  Thus we can write \eqref{E2.1} as follows.

\begin{equation}\label{E2.3}
\begin{split}
D_{n,k}(1,y(1-y))&=k\,\Big[ \displaystyle\frac{y^n(1-y)-y(1-y)^n}{2y-1}\Big]+D_n(1,y(1-y)),
\end{split}
\end{equation}

\noindent where $y\neq \frac{1}{2}$. 

When $y=\frac{1}{2}$, we have

\begin{equation}\label{E2.6}
D_{n,k}\Big(1,\frac{1}{4}\Big)\,=\,\frac{k(n-1)+2}{2^n};
\end{equation}

\noindent see \cite[Eq. 2.6]{Fernando-2016-3}.

\begin{prop}\label{P2.2}(see \cite[Proposition~2.6]{Fernando-2016-3}) Let $p$ be an odd prime and $n$ be a non-negative integer. Then 
$$D_{0,k}(1,x)=2-k,\,\,D_{1,k}(1,x)=1,\,\,\textnormal{and}$$
$$D_{n,k}(1,x)=D_{n-1,k}(1,x)-x\,D_{n-2,k}(1,x),\,\, \textnormal{for} \,\,n\geq 2.$$
\end{prop}

Now we list a few results that appeared in \cite{Fernando-2016-3}. 

For all $x\in \f_q$, we have 

\begin{equation}\label{E2.7}
\begin{split}
D_{p^l,k}(1,x)&=\frac{k}{2}\,(1-4x)^{\frac{p^l-1}{2}}+1-\frac{k}{2},
\end{split}
\end{equation}

\begin{equation}\label{E2.9}
\begin{split}
D_{p^l+1,k}(1,x)&=\Big(\frac{1}{2} - \frac{k}{4}\Big)\,(1-4x)^{\frac{p^l+1}{2}}+\frac{k}{4}\,(1-4x)^{\frac{p^l-1}{2}} +\frac{1}{2},
\end{split}
\end{equation}

and

\begin{equation}\label{E2.10}
\begin{split}
D_{p^l+2,k}(1,x)&=\frac{1}{2}\,(1-4x)^{\frac{p^l+1}{2}}+\frac{k}{2}\,x\,(1-4x)^{\frac{p^l-1}{2}}-\Big(1-\frac{k}{2}\Big)x+\frac{1}{2};
\end{split}
\end{equation}

\noindent see \cite[Eq. 2.7]{Fernando-2016-3}, \cite[Eq. 2.9]{Fernando-2016-3}, and \cite[Eq. 2.10]{Fernando-2016-3}, respectively.

Now we list a very well-known result on permutation monomials over finite fields. 

\begin{lem}\label{L2.1}(see \cite{Lidl-Niederreiter-97})
The monomial $x^n$ is a PP of $\f_q$ if and only if $(n, q-1)=1.$
\end{lem}

The following result appeared in \cite{Hou-Mullen-Sellers-Yucas-FFA-2009}.

\begin{lem}(see \cite[Lemma~5.5]{Hou-Mullen-Sellers-Yucas-FFA-2009})
Let $l$ be a positive odd integer and let $n=\frac{3^l+1}{2}$. Then in $\f_3[x]$, 
$$D_n(1, 1-x^2)= - D_n(x,1).$$
\end{lem}


\section{The case $n=p^l+3$}

Let $n=p^l+3$. Then from \eqref{E2.9}, \eqref{E2.10}, and Proposition~\ref{P2.2} we have

\begin{equation}\label{E3.1}
\begin{split}
&D_{p^l+3,k}(1,x)=D_{p^l+2,k}(1,x)-xD_{p^l+1,k}(1,x)\cr
&= \frac{1}{2}\,(1-4x)^{\frac{p^l+1}{2}}+\frac{k}{2}\,x\,(1-4x)^{\frac{p^l-1}{2}}-\Big(1-\frac{k}{2}\Big)x+\frac{1}{2}\cr
&-\Big(\frac{1}{2}-\frac{k}{4}\Big)\,x\,(1-4x)^{\frac{p^l+1}{2}}-\frac{k}{4}\,x\,(1-4x)^{\frac{p^l-1}{2}}-\frac{1}{2}\,x.
\end{split}
\end{equation}

Let $u=1-4x$. Then

\begin{equation}\label{E3.2}
\begin{split}
&D_{p^l+3,k}(1,x)\cr
&= \frac{1}{2}\,u^{\frac{p^l+1}{2}}+\frac{k}{2}\,\Big(\frac{1-u}{4}\Big)\,u^{\frac{p^l-1}{2}}-\Big(1-\frac{k}{2}\big)\Big(\frac{1-u}{4}\Big)+\frac{1}{2}\cr
&-\Big(\frac{1}{2}-\frac{k}{4}\Big)\,\Big(\frac{1-u}{4}\Big)\,u^{\frac{p^l+1}{2}}-\frac{k}{4}\,\Big(\frac{1-u}{4}\Big)\,u^{\frac{p^l-1}{2}}-\frac{1}{2}\,\Big(\frac{1-u}{4}\Big)\cr
&=\frac{(2-k)}{16}\,u^{\frac{p^l+3}{2}}+\frac{3}{8}\,u^{\frac{p^l+1}{2}}+\frac{k}{16}\,u^{\frac{p^l-1}{2}}+\frac{(3-k)}{8}\,u+\frac{(k+1)}{8}.
\end{split}
\end{equation}

This leads to the following immediate theorem.

\begin{thm} 
$D_{p^l+3,k}(1,x)$ is a PP of $\f_{p^e}$ if and only if the polynomial 
$$(2-k)\,x^{\frac{p^l+3}{2}}+6\,x^{\frac{p^l+1}{2}}+k\,x^{\frac{p^l-1}{2}}+2(3-k)\,x$$
is a PP of $\f_{p^e}$. 
\end{thm} 

Let
\begin{equation}\label{N1}
f(x)=(2-k)\,x^{\frac{p^l+3}{2}}+6\,x^{\frac{p^l+1}{2}}+k\,x^{\frac{p^l-1}{2}}+2(3-k)\,x.
\end{equation}

\subsection{The case $p=3$} \hspace*{\fill} \\

Note that in this case $k=0, 1,$ or $2$. 

\begin{thm}
$D_{3^l+3,0}(1,x)$ is a PP of $\f_{p^e}$ if and only if $\textnormal{gcd}(\frac{3^l+3}{2}, 3^e-1)=1$. 
\end{thm}

\begin{proof}
When $k=0$,  \eqref{N1} becomes 
$$f(x)=2x^{\frac{3^l+3}{2}}.$$
The rest is obvious. 
\end{proof}

\begin{thm}
$D_{3^l+3,1}(1,x)$ is not a PP of $\f_{p^e}$.
\end{thm}

\begin{proof}
When $p=3$ and $k=1$, $f(x)$ is a PP of $\f_{3^e}$ if and only if the trinomial $x^{\frac{p^l+3}{2}}+x^{\frac{p^l-1}{2}}+x$ is a PP of $\f_{3^e}$. But $x^{\frac{p^l+3}{2}}+x^{\frac{p^l-1}{2}}+x$ is clearly not a PP of $\f_{3^e}$.
\end{proof}

\begin{thm} (\cite[Theorem~4.1]{Fernando-2016-4})
$D_{3^l+3,2}(1,x)$ is a PP of $\f_{p^e}$ if and only if 
\begin{itemize}
\item [(i)]  $l=0$, or
\item [(ii)] $l=me+1$, where $m$ is a non-negative even integer.
\end{itemize}
\end{thm}

\subsection{The case $p>3$} \hspace*{\fill} \\

We have the following results in this case. 

\begin{rmk}
Let $p\geq 5$ and $k=2$. If $l=0$, then clearly $D_{p^l+3,k}(1,x)$ is a PP of $\f_{p^e}$ since every liner polynomial is a PP. 
\end{rmk}

\begin{thm}
Let $k=2$ and $p\equiv 1\pmod{4}$. If $l\neq 0$, then $D_{p^l+3,k}(1,x)$ is not a PP of $\f_{p^e}$. 
\end{thm}

\begin{proof}
Let $k=2$. Then $D_{p^l+3,k}(1,x)$ is a PP of $\f_{p^e}$ if and only if $f(x)=3x^{\frac{p^l+1}{2}}+x^{\frac{p^l-1}{2}}+x$ is a PP of $\f_{p^e}$. Since $p\equiv 1\pmod{4}$, $(-\frac{1}{4})$ is a quadratic residue modulo $p$. The proof follows from the fact that $f(0)=0=f\Big(-\frac{1}{4}\Big)$. 
\end{proof}




\begin{rmk}
Let $p>7$, $k=7$. Then $D_{p^l+3,k}(1,x)$ is not a PP of $\f_{p^e}$ since $f(0)=0=f(1)$. 
\end{rmk}

\begin{rmk}
Let $p>5$, $k=0$, and $-6$ be a quadratic residue of $p$. Then $D_{p^l+3,k}(1,x)$ is not a PP of $\f_{p^e}$ since $f(0)=0=f(-6)$. 
\end{rmk}


\section{The case $n=p^{l_1}+p^{l_2}+p^{l_3}$}

In this Section, we explain the case $n=p^{l_1}+p^{l_2}+p^{l_3}$, where $l_1, l_2$, and $l_3$ are non-negative integers. 

Let $n=p^{l_1}+p^{l_2}+p^{l_3}$ in \eqref{E2.3}  and \eqref{E2.6}. Then for all $x\in \f_q$, we have 

\begin{equation}
\begin{split}
&D_{n,k}(1,x)\cr
&=\frac{k}{8}\,(1-4x)^\frac{p^{l_1}+p^{l_2}+p^{l_3}-1}{2}+\frac{(2-k)}{8}\,[(1-4x)^\frac{p^{l_1}+p^{l_2}}{2}+(1-4x)^\frac{p^{l_1}+p^{l_3}}{2}+(1-4x)^\frac{p^{l_2}+p^{l_3}}{2}]\cr
&+\frac{k}{8}[(1-4x)^\frac{p^{l_1}-1}{2}+(1-4x)^\frac{p^{l_2}-1}{2}+(1-4x)^\frac{p^{l_3}-1}{2}]+\frac{(2-k)}{8};
\end{split}
\end{equation}

\noindent see \cite[Section 4]{Fernando-2017-II} for detailed computation. 

This leads to the following immediate theorem. 

\begin{thm}
$D_{p^{l_1}+p^{l_2}+p^{l_3},k}(1,x)$ is a PP of $\f_q$ if and only if the seven-term polynomial
\[ 
\begin{split}
k\,x^\frac{p^{l_1}+p^{l_2}+p^{l_3}-1}{2}+(2-k)\,[x^\frac{p^{l_1}+p^{l_2}}{2}+x^\frac{p^{l_1}+p^{l_3}}{2}+x^\frac{p^{l_2}+p^{l_3}}{2}]+k[x^\frac{p^{l_1}-1}{2}+x^\frac{p^{l_2}-1}{2}+x^\frac{p^{l_3}-1}{2}]
\end{split}
\]
is a PP of $\f_q$. 
\end{thm}

Let 

\[ 
\begin{split}
f(x)=k\,x^\frac{p^{l_1}+p^{l_2}+p^{l_3}-1}{2}+(2-k)\,[x^\frac{p^{l_1}+p^{l_2}}{2}+x^\frac{p^{l_1}+p^{l_3}}{2}+x^\frac{p^{l_2}+p^{l_3}}{2}]+k[x^\frac{p^{l_1}-1}{2}+x^\frac{p^{l_2}-1}{2}+x^\frac{p^{l_3}-1}{2}].
\end{split}
\]

\subsection{The case $p=3$}  \hspace*{\fill} \\ 

In this case, $k=0, 1,\,\textnormal{or}\,\,2$. 

\begin{thm}
Let $k=0$, $p=3$, and $n=p^{l_1}+p^{l_2}+p^{l_3}$, where $l_1, l_2$, and $l_3$ are non-negative integers. Then $D_{n,k}(1,x)$ is not a PP of $\f_{p^e}$. 
\end{thm}

\begin{proof}

\textbf{Case 1.} $l_1=l_2=l_3=0$. $f(x)=0$ which is not a PP of $\f_{3^e}$. 

\textbf{Case 2.} Exactly two of $l_1, l_2, l_3$ are zero. Without loss of generality, let $l_1\neq 0$ and $l_2=l_3=0$. Then $f(x)$ is a PP of $\f_{3^e}$ if and only if the trinomial $x^\frac{p^{l_1}+1}{2}+x^\frac{p^{l_1}+1}{2}+x$ is a PP of $\f_{3^e}$. Clearly, $x^\frac{p^{l_1}+1}{2}+x^\frac{p^{l_1}+1}{2}+x$ is not a PP of $\f_{3^e}$. 

\textbf{Case 3.} Exactly one of $l_1, l_2, l_3$ is zero. Without loss of generality, let $l_1\neq 0$, $l_2\neq 0$ and $l_3=0$. Then $f(x)$ is a PP of $\f_{3^e}$ if and only if the trinomial $x^\frac{p^{l_1}+p^{l_2}}{2}+x^\frac{p^{l_1}+1}{2}+x^\frac{p^{l_2}+1}{2}$ is a PP of $\f_{3^e}$. Clearly, $x^\frac{p^{l_1}+p^{l_2}}{2}+x^\frac{p^{l_1}+1}{2}+x^\frac{p^{l_2}+1}{2}$ is not a PP of $\f_{3^e}$. 

\textbf{Case 4.} $l_1\neq 0$, $l_2\neq 0$ and $l_3\neq 0$. Then $f(x)$ is a PP of $\f_{3^e}$ if and only if the trinomial $x^\frac{p^{l_1}+p^{l_2}}{2}+x^\frac{p^{l_1}+p^{l_3}}{2}+x^\frac{p^{l_2}+p^{l_3}}{2}$ is a PP of $\f_{3^e}$. Clearly, $x^\frac{p^{l_1}+p^{l_2}}{2}+x^\frac{p^{l_1}+p^{l_3}}{2}+x^\frac{p^{l_2}+p^{l_3}}{2}$ is not a PP of $\f_{3^e}$.  

\end{proof}

\begin{rmk}
The case in which exactly two of $l_1, l_2$ and $l_3$ are zero and the other is equal to $e$, i.e. $n=p^e+2$, is explained \cite[Corollary~5.2]{Hou-Mullen-Sellers-Yucas-FFA-2009}. It is shown that $D_{p^e+2,0}(1,x)$ is a PP of $\f_{p^e}$ if and only if $p^e\equiv 1 \pmod{3}$.
\end{rmk}

\begin{thm}
Let $k=1$, $p=3$, and $n=p^{l_1}+p^{l_2}+p^{l_3}$. Assume that exactly one of $l_1, l_2, l_3$ is zero. Then $D_{n,k}(1,x)$ is not a PP of $\f_{p^e}$. 
\end{thm}

\begin{proof}
Without loss of generality, let $l_1\neq 0$, $l_2\neq 0$ and $l_3=0$. Then $f(x)$ is a PP of $\f_{p^e}$ if and only if $$g(x)=2x^\frac{p^{l_1}+p^{l_2}}{2}+x^\frac{p^{l_1}+1}{2}+x^\frac{p^{l_2}+1}{2}+x^\frac{p^{l_1}-1}{2}+x^\frac{p^{l_2}-1}{2}$$ is a PP of $\f_{p^e}$. Clearly, $g(x)$ is not a PP of $\f_{p^e}$. 
\end{proof}

\begin{rmk}
Let $k=1$ and $p=3$. The cases in which $l_1=l_2=l_3=0$ and exactly two of $l_1, l_2, l_3$ are zero, are explained in \cite[Theorem~4.1]{Fernando-2016-4}.
\end{rmk}

\begin{thm}
Let $k=2$, $p=3$, and $n=p^{l_1}+p^{l_2}+p^{l_3}$. Assume that exactly one of $l_1, l_2, l_3$ is zero. Then $D_{n,k}(1,x)$ is not a PP of $\f_{p^e}$. 
\end{thm}

\begin{proof}
Without loss of generality, let $l_1\neq 0$, $l_2\neq 0$ and $l_3=0$. Then $f(x)$ is a PP of $\f_{p^e}$ if and only if $$g(x)=x^\frac{p^{l_1}+p^{l_2}}{2}+x^\frac{p^{l_1}-1}{2}+x^\frac{p^{l_2}-1}{2}$$ is a PP of $\f_{p^e}$. Clearly, $g(x)$ is not a PP of $\f_{p^e}$. 
\end{proof}

\begin{rmk}
Let $k=2$ and $p=3$. The cases in which $l_1=l_2=l_3=0$ and exactly two of $l_1, l_2, l_3$ are zero, are explained in \cite[Theorem~2.15]{Fernando-2016-3}.
\end{rmk}

\subsection{The case $p>3$} 

\begin{thm}
Let $p\geq 5$ and $k\neq 0,3$. Assume that 
\begin{enumerate}
\item exactly two of $l_1, l_2,$ and $ l_3$ are non-zero, and 
\item $\displaystyle{\frac{k}{(k-3)}}$ is a quadratic residue modulo $p$.
\end{enumerate} 
Then $D_{n,k}(1,x)$ is not a PP of $\f_{p^e}$.
\end{thm}

\begin{proof}
Since exactly two of $l_1, l_2,$ and $ l_3$ are non-zero, we have $f(0)=k$. The proof follows from the fact that $$f\Big(\frac{k}{(k-3)}\Big)=k=f(0).$$
\end{proof}

\begin{thm}
Let $p\geq 5$ and $k\neq 0,3$. Assume that 
\begin{enumerate}
\item $l_1, l_2,$ and $ l_3$ are all non-zero, and 
\item $\displaystyle{\frac{3k}{2(k-3)}}$ is a quadratic residue modulo $p$.
\end{enumerate} 
Then $D_{n,k}(1,x)$ is not a PP of $\f_{p^e}$.
\end{thm}

\begin{proof}
Since $l_1, l_2,$ and $ l_3$ are all non-zero, we have $f(0)=0$. The proof follows from the fact that $$f\Big(\frac{3k}{2(k-3)}\Big)=0=f(0).$$
\end{proof}

\begin{rmk}
The case in which exactly one of $l_1, l_2,$ and $ l_3$ is nonzero is explained in \cite{Fernando-2016-3} and \cite{Fernando-2016-4}. 
\end{rmk}



\begin{rmk}
Let $p>3$ and $k=3$. Then $D_{p^{l_1}+p^{l_2}+p^{l_3},k}(1,x)$ is not a PP of $\f_{p^e}$ for any $l_1, l_2$, and $l_3$ since $f(a)=9$ for any quadratic residue $a$ of $p$. 
\end{rmk}


\section{The case $n=p^{l_1}+p^{l_2}+p^{l_3}+p^{l_4}$}

In this Section, we explain the case $n=p^{l_1}+p^{l_2}+p^{l_3}+p^{l_4}$, where $l_1, l_2, l_3$, and $l_4$ are non-negative integers. 

Let $n=p^{l_1}+p^{l_2}+p^{l_3}+p^{l_4}$. Then for all $x\in \f_q$, we have 

\begin{equation}
\tiny
{
\begin{split}
&D_{n,k}(1,x)\cr
&=\frac{(2-k)}{16}\,(1-4x)^\frac{p^{l_1}+p^{l_2}+p^{l_3}+p^{l_4}}{2}\cr 
&+\frac{k}{16}\ [(1-4x)^\frac{p^{l_1}+p^{l_2}+p^{l_3}-1}{2}+(1-4x)^\frac{p^{l_1}+p^{l_2}+p^{l_4}-1}{2}+(1-4x)^\frac{p^{l_1}+p^{l_3}+p^{l_4}-1}{2}+(1-4x)^\frac{p^{l_2}+p^{l_3}+p^{l_4}-1}{2}] \cr
&+\frac{(2-k)}{16}\,[(1-4x)^\frac{p^{l_1}+p^{l_2}}{2}+(1-4x)^\frac{p^{l_1}+p^{l_3}}{2}+(1-4x)^\frac{p^{l_2}+p^{l_3}}{2}+(1-4x)^\frac{p^{l_1}+p^{l_4}}{2}+(1-4x)^\frac{p^{l_2}+p^{l_4}}{2}+(1-4x)^\frac{p^{l_3}+p^{l_4}}{2}]\cr
&+\frac{k}{16}\,[(1-4x)^\frac{p^{l_1}-1}{2}+(1-4x)^\frac{p^{l_2}-1}{2}+(1-4x)^\frac{p^{l_3}-1}{2}+(1-4x)^\frac{p^{l_4}-1}{2}]+\frac{(2-k)}{16};
\end{split}
}
\end{equation}

\noindent see \cite[Section 5]{Fernando-2017-II} for detailed computation. 

This leads to the following immediate theorem. 

\begin{thm}\label{LNC}
$D_{p^{l_1}+p^{l_2}+p^{l_3}+p^{l_4},k}(1,x)$ is a PP of $\f_q$ if and only if the polynomial
\[ 
\tiny
{
\begin{split}
&(2-k)\,x^\frac{p^{l_1}+p^{l_2}+p^{l_3}+p^{l_4}}{2}+k\ [x^\frac{p^{l_1}+p^{l_2}+p^{l_3}-1}{2}+x^\frac{p^{l_1}+p^{l_2}+p^{l_4}-1}{2}+x^\frac{p^{l_1}+p^{l_3}+p^{l_4}-1}{2}+x^\frac{p^{l_2}+p^{l_3}+p^{l_4}-1}{2}] \cr
&+(2-k)\,[x^\frac{p^{l_1}+p^{l_2}}{2}+x^\frac{p^{l_1}+p^{l_3}}{2}+x^\frac{p^{l_2}+p^{l_3}}{2}+x^\frac{p^{l_1}+p^{l_4}}{2}+x^\frac{p^{l_2}+p^{l_4}}{2}+x^\frac{p^{l_3}+p^{l_4}}{2}]+k\,[x^\frac{p^{l_1}-1}{2}+x^\frac{p^{l_2}-1}{2}+x^\frac{p^{l_3}-1}{2}+x^\frac{p^{l_4}-1}{2}]
\end{split}
}
\]
is a PP of $\f_q$. 
\end{thm}

Let

\[ 
\tiny
{
\begin{split}
h(x)&=(2-k)\,x^\frac{p^{l_1}+p^{l_2}+p^{l_3}+p^{l_4}}{2}+k\ [x^\frac{p^{l_1}+p^{l_2}+p^{l_3}-1}{2}+x^\frac{p^{l_1}+p^{l_2}+p^{l_4}-1}{2}+x^\frac{p^{l_1}+p^{l_3}+p^{l_4}-1}{2}+x^\frac{p^{l_2}+p^{l_3}+p^{l_4}-1}{2}] \cr
&+(2-k)\,[x^\frac{p^{l_1}+p^{l_2}}{2}+x^\frac{p^{l_1}+p^{l_3}}{2}+x^\frac{p^{l_2}+p^{l_3}}{2}+x^\frac{p^{l_1}+p^{l_4}}{2}+x^\frac{p^{l_2}+p^{l_4}}{2}+x^\frac{p^{l_3}+p^{l_4}}{2}]\cr
&+k\,[x^\frac{p^{l_1}-1}{2}+x^\frac{p^{l_2}-1}{2}+x^\frac{p^{l_3}-1}{2}+x^\frac{p^{l_4}-1}{2}].
\end{split}
}
\]

\begin{rmk}
The cases in which $l_1=l_2=l_3=l_4=0$ and exactly three of $l_1, l_2, l_3$, and $l_4$ are zero, are completely covered in Section 3 of the present paper. 
\end{rmk}

\subsection{The case $p=3$, $k\neq 2$, $l_1\neq 0$, $l_2\neq 0$, $l_3=l_4=0$} \hspace*{\fill} \\

We first consider the case in which exactly two of $l_1, l_2, l_3$, and $l_4$ are zero. Let's consider the case $k=0$ first. 

\begin{thm}
Let $k=0$, $p=3$, and $n=p^{l_1}+p^{l_2}+p^{l_3}+p^{l_4}$. Assume that exactly two of $l_1, l_2, l_3$, and $l_4$ are zero and the two non zero $l_i$ have the same parity. Then $D_{n,k}(1,x)$ is not a PP of $\f_{p^e}$.
\end{thm}

\begin{proof}
Without loss of generality, let $l_1\neq 0$, $l_2\neq 0$ and $l_3=0=l_4$. Since $k=0$, note that $D_{n,k}(1,x)$ is a PP of $\f_{p^e}$ if and only if 
$$g(x)=x^{\frac{p^{l_1}+p^{l_2}}{2}+1}+x^\frac{p^{l_1}+p^{l_2}}{2}+2x^\frac{p^{l_1}+1}{2}+2x^\frac{p^{l_2}+1}{2}+x$$ 
is a PP of $\f_{p^e}$.

Now we claim that $g(x)$ is not a PP of $\f_{p^e}$. Note that $g(0)=0$ and $g(1)=1$. 

\textbf{Case 1.} Both $l_1$ and $l_2$ are even. 

Since both $l_1$ and $l_2$ are even, $\frac{3^{l_1}+1}{2}$ and $\frac{3^{l_2}+1}{2}$ are odd. Then $g(-1)=1$.

\textbf{Case 2.} Both $l_1$ and $l_2$ are odd. 

Since both $l_1$ and $l_2$ are odd, $\frac{3^{l_1}+1}{2}$ and $\frac{3^{l_2}+1}{2}$ are even. Then $g(-1)=0$.

This completes the proof. 

\end{proof}

We have the following immediate theorem. 

\begin{thm}
Let $k=0$, $p=3$, and $n=p^{l_1}+p^{l_2}+p^{l_3}+p^{l_4}$. Assume that exactly two of $l_1, l_2, l_3$, and $l_4$ are zero and the two non zero $l_i$ have different parity. Then $D_{n,k}(1,x)$ is a PP of $\f_{p^e}$ if and only if 
$$g(x)=x^{\frac{p^{l_1}+p^{l_2}}{2}+1}+x^\frac{p^{l_1}+p^{l_2}}{2}+2x^\frac{p^{l_1}+1}{2}+2x^\frac{p^{l_2}+1}{2}+x$$ 
is a PP of $\f_{p^e}$.
\end{thm}

Now let's consider the case where $k=1$. 

\begin{thm}
Let $k=1$, $p=3$, and $n=p^{l_1}+p^{l_2}+p^{l_3}+p^{l_4}$. Assume that exactly two of $l_1, l_2, l_3$, and $l_4$ are zero and the two non-zero $l_i$ are both odd or they have different parity. Then $D_{n,k}(1,x)$ is not a PP of $\f_{p^e}$.
\end{thm}

\begin{proof}
Without loss of generality, let $l_1\neq 0$, $l_2\neq 0$ and $l_3=0=l_4$. Since $k=1$, note that $D_{n,k}(1,x)$ is a PP of $\f_{p^e}$ if and only if 
$$g(x)=x^{\frac{p^{l_1}+p^{l_2}}{2}+1}+x^\frac{p^{l_1}-1}{2}+x^\frac{p^{l_2}-1}{2}+x$$ 
is a PP of $\f_{p^e}$.

Now we claim that $g(x)$ is not a PP of $\f_{p^e}$. Note that $g(0)=0$ and $g(1)=1$. The proof follows from the fact that  $g(-1)=1$ when $l_1$ and $l_2$ have different parity, and when both $l_1$ and $l_2$ are odd. 

\end{proof}

\begin{thm}
Let $k=1$, $p=3$, and $n=p^{l_1}+p^{l_2}+p^{l_3}+p^{l_4}$. Assume that exactly two of $l_1, l_2, l_3$, and $l_4$ are zero and the two non-zero $l_i$ are both even. Then $D_{n,k}(1,x)$ is a PP of $\f_{p^e}$ if and only if 
$$g(x)=x^{\frac{p^{l_1}+p^{l_2}}{2}+1}+x^\frac{p^{l_1}-1}{2}+x^\frac{p^{l_2}-1}{2}+x$$ 
is a PP of $\f_{p^e}$.
\end{thm}

\subsection{The case $p=3$, $k\neq 2$,  $l_1\neq 0$, $l_2\neq 0$, $l_3\neq 0$, $l_4=0$}  \hspace*{\fill} \\

We now consider the case in which exactly one of $l_1, l_2, l_3$, and $l_4$ is zero. First, consider the case where $k=0$. 

\begin{thm}
Let $k=0$, $p=3$, and $n=p^{l_1}+p^{l_2}+p^{l_3}+p^{l_4}$. Assume that exactly one of $l_1, l_2, l_3$, and $l_4$ is zero. Assume that non-zero $l_i$ are all even or exactly two are odd. Then $D_{n,k}(1,x)$ is not a PP of $\f_{p^e}$.
\end{thm}

\begin{proof}
Without loss of generality, let $l_1\neq 0$, $l_2\neq 0$, $l_3\neq 0$ and $l_4=0$. Since $k=0$, note that $D_{n,k}(1,x)$ is a PP of $\f_{p^e}$ if and only if 
$$g(x)=x^{\frac{p^{l_1}+p^{l_2}+p^{l_3}+1}{2}}+x^{\frac{p^{l_1}+p^{l_2}}{2}}+x^{\frac{p^{l_1}+p^{l_3}}{2}}+x^{\frac{p^{l_2}+p^{l_3}}{2}}+x^\frac{p^{l_1}+1}{2}+x^\frac{p^{l_2}+1}{2}+x^\frac{p^{l_3}+1}{2}$$ 
is a PP of $\f_{p^e}$.

Now we claim that $g(x)$ is not a PP of $\f_{p^e}$. Note that $g(0)=0$ and $g(1)=1$. The proof follows from the fact that $g(-1)=1$ when $l_1$, $l_2$, and $l_3$ are all even, and $g(-1)=0$ when exactly two of $l_1$, $l_2$, and $l_3$ are odd.

\end{proof}

Now let's consider the case where $k=1$. 

\begin{thm}
Let $k=1$, $p=3$, and $n=p^{l_1}+p^{l_2}+p^{l_3}+p^{l_4}$. Assume that exactly one of $l_1, l_2, l_3$, and $l_4$ is zero. Then $D_{n,k}(1,x)$ is not a PP of $\f_{p^e}$.
\end{thm}

\begin{proof}
Without loss of generality, let $l_1\neq 0$, $l_2\neq 0$, $l_3\neq 0$ and $l_4=0$. Since $k=1$, $D_{n,k}(1,x)$ is a PP of $\f_{p^e}$ if and only if 
\[
\begin{split}
&g(x)=x^{\frac{p^{l_1}+p^{l_2}+p^{l_3}+1}{2}}+x^{\frac{p^{l_1}+p^{l_2}+p^{l_3}-1}{2}}+2x^{\frac{p^{l_1}+p^{l_2}}{2}}+2x^{\frac{p^{l_1}+p^{l_3}}{2}}+2x^{\frac{p^{l_2}+p^{l_3}}{2}}+x^\frac{p^{l_1}+1}{2}+x^\frac{p^{l_2}+1}{2}+x^\frac{p^{l_3}+1}{2}\cr
&+x^\frac{p^{l_1}-1}{2}+x^\frac{p^{l_2}-1}{2}+x^\frac{p^{l_3}-1}{2}
\end{split}
\]
is a PP of $\f_{p^e}$.

Now we claim that $g(x)$ is not a PP of $\f_{p^e}$. Note that $g(0)=0$ and $g(1)=2$. 

\textbf{Case 1.} $l_1$, $l_2$, and $l_3$ are all even. Then $g(-1)=0$. 

\textbf{Case 2.} $l_1$, $l_2$, and $l_3$ are all odd. Then $g(-1)=0$.

\textbf{Case 3.} Exactly one of $l_1, l_2$, and $l_3$ is odd. Then $g(-1)=2$. 

\textbf{Case 4.} Exactly two of $l_1, l_2$, and $l_3$ are odd. Then $g(-1)=2$. 

This completes the proof. 

\end{proof}

\subsection{The case $k\neq 2$,  $l_1l_2l_3l_4\neq 0$}  \hspace*{\fill} \\

Now we consider the case in which $l_1\neq 0, l_2 \neq 0, l_3\neq 0$, and $l_4\neq 0$. First, consider the case where $k=0$.

\begin{thm}
Let $k=0$, $p=3$, and $n=p^{l_1}+p^{l_2}+p^{l_3}+p^{l_4}$. Assume that $l_1\neq 0, l_2 \neq 0, l_3\neq 0$, and $l_4\neq 0$. If
\begin{enumerate}
\item all $l_1, l_2, l_3$, and $l_4$ are odd,
\item all $l_1, l_2, l_3$, and $l_4$ are even, or
\item exactly two of $l_1, l_2, l_3$, and $l_4$ are odd, 
\end{enumerate}
then $D_{n,k}(1,x)$ is not a PP of $\f_{p^e}$.
\end{thm}

\begin{proof}
Assume that $l_1\neq 0, l_2 \neq 0, l_3\neq 0$, and $l_4\neq 0$. Since $k=0$, $D_{n,k}(1,x)$ is a PP of $\f_{p^e}$ if and only if 
\[
\begin{split}
&g(x)=x^{\frac{p^{l_1}+p^{l_2}+p^{l_3}+p^{l_4}}{2}}+x^{\frac{p^{l_1}+p^{l_2}}{2}}+x^{\frac{p^{l_1}+p^{l_3}}{2}}+x^{\frac{p^{l_2}+p^{l_3}}{2}}+x^{\frac{p^{l_1}+p^{l_4}}{2}}+x^\frac{p^{l_2}+p^{l_4}}{2}+x^\frac{p^{l_3}+p^{l_4}}{2}
\end{split}
\]
is a PP of $\f_{p^e}$.

Now we claim that $g(x)$ is not a PP of $\f_{p^e}$. Note that $g(0)=0$ and $g(1)=1$. 

\textbf{Case 1.} $l_1$, $l_2$, $l_3$, and $l_4$ are all even. Then $g(-1)=1$. 

\textbf{Case 2.} $l_1$, $l_2$, $l_3$, and $l_4$ are all odd.  Then $g(-1)=1$.

\textbf{Case 3.} Exactly two of $l_1$, $l_2$, $l_3$, and $l_4$ are odd. Then $g(-1)=0$. 

This completes the proof. 

\end{proof}

Let's consider the case where $k=1$. 

\begin{thm}
Let $k=1$, $p=3$, and $n=p^{l_1}+p^{l_2}+p^{l_3}+p^{l_4}$. If $l_1\neq 0, l_2 \neq 0, l_3\neq 0$, and $l_4\neq 0$, then $D_{n,k}(1,x)$ is not a PP of $\f_{p^e}$.
\end{thm}

\begin{proof}
Assume that $l_1\neq 0, l_2 \neq 0, l_3\neq 0$, and $l_4\neq 0$. Since $k=1$, $D_{n,k}(1,x)$ is a PP of $\f_{p^e}$ if and only if 
\[ 
\tiny
{
\begin{split}
g(x)&=x^\frac{p^{l_1}+p^{l_2}+p^{l_3}+p^{l_4}}{2}+x^\frac{p^{l_1}+p^{l_2}+p^{l_3}-1}{2}+x^\frac{p^{l_1}+p^{l_2}+p^{l_4}-1}{2}+x^\frac{p^{l_1}+p^{l_3}+p^{l_4}-1}{2}+x^\frac{p^{l_2}+p^{l_3}+p^{l_4}-1}{2} \cr
&+x^\frac{p^{l_1}+p^{l_2}}{2}+x^\frac{p^{l_1}+p^{l_3}}{2}+x^\frac{p^{l_2}+p^{l_3}}{2}+x^\frac{p^{l_1}+p^{l_4}}{2}+x^\frac{p^{l_2}+p^{l_4}}{2}+x^\frac{p^{l_3}+p^{l_4}}{2}\cr
&+x^\frac{p^{l_1}-1}{2}+x^\frac{p^{l_2}-1}{2}+x^\frac{p^{l_3}-1}{2}+x^\frac{p^{l_4}-1}{2}
\end{split}
}
\]
is a PP of $\f_{p^e}$. But, clearly $g(x)$ is not a PP of $\f_{p^e}$ since $g(0)=0=g(1)$. 
\end{proof}

Let $p>3$ and $a$ be a quadratic residue of $p$. Then $a=\beta^2$ for some $1\leq \beta \leq p-1$. We have $$h(a)=h(\beta^2)=(2-k)\beta^4+2(6-k)\beta^2+4k=(2-k)a^2+2(6-k)a+4k.$$ Then we have the following.

\begin{thm}
Let $p\geq 5$, $0\leq k\leq p-1$ with $k\neq 2$. Assume that 
\begin{enumerate}
\item exactly one of $l_1, l_2, l_3$, and $l_4$ is non-zero, and 
\item $a$ is a quadratic residue modulo $p$ which satisfies $(2-k)a^2+2(6-k)a+k=0$.
\end{enumerate} 
Then $D_{n,k}(1,x)$ is not a PP of $\f_{p^e}$.
\end{thm}

\begin{proof}
Since exactly one of $l_1, l_2, l_3$, and $l_4$ is non-zero, we have $h(0)=3k$. The proof follows from the fact that 
$$h(a)=(2-k)a^2+2(6-k)a+4k=3k=h(0).$$
\end{proof}

\begin{thm}
Let $p\geq 5$, $0\leq k\leq p-1$ with $k\neq 2$. Assume that 
\begin{enumerate}
\item exactly two of $l_1, l_2, l_3$, and $l_4$ are non-zero, and 
\item $a$ is a quadratic residue modulo $p$ which satisfies $(2-k)a^2+2(6-k)a+2k=0$.
\end{enumerate} 
Then $D_{n,k}(1,x)$ is not a PP of $\f_{p^e}$.
\end{thm}

\begin{proof}
Since exactly two of $l_1, l_2, l_3$, and $l_4$ are non-zero, we have $h(0)=2k$. The proof follows from the fact that 
$$h(a)=(2-k)a^2+2(6-k)a+4k=2k=h(0).$$
\end{proof}

\begin{thm}
Let $p\geq 5$, $0\leq k\leq p-1$ with $k\neq 2$. Assume that 
\begin{enumerate}
\item exactly three of $l_1, l_2, l_3$, and $l_4$ are non-zero, and
\item $a$ is a quadratic residue modulo $p$ which satisfies $(2-k)a^2+2(6-k)a+3k=0$.
\end{enumerate} 
Then $D_{n,k}(1,x)$ is not a PP of $\f_{p^e}$.
\end{thm}

\begin{proof}
Since exactly three of $l_1, l_2, l_3$, and $l_4$ are non-zero, we have $h(0)=k$. The proof follows from the fact that 
$$h(a)=(2-k)a^2+2(6-k)a+4k=k=h(0).$$
\end{proof}

\begin{thm}
Let $p\geq 5$, $0\leq k\leq p-1$ with $k\neq 2$. Assume that 
\begin{enumerate}
\item $l_1, l_2, l_3$, and $l_4$ are non-zero, and 
\item $a$ is a quadratic residue modulo $p$ which satisfies $(2-k)a^2+2(6-k)a+4k=0$. 
\end{enumerate} 
Then $D_{n,k}(1,x)$ is not a PP of $\f_{p^e}$.
\end{thm}

\begin{proof}
Since $l_1, l_2, l_3$, and $l_4$ are all non-zero, we have $h(0)=0$. Since $a$ is a quadratic residue modulo $p$ such that the equation $(2-k)a^2+2(6-k)a+4k=0$ has a solution in $\mathbb{F}_p$, the proof follows from the fact that 
$$h(a)=0=h(0).$$
\end{proof}

\begin{thm}
Let $p\geq 5$, $0\leq k\leq p-1$ with $k\neq 2$. Assume that $l_1=l_2=l_3=l_4=0$. Then $D_{n,k}(1,x)$ is not a PP of $\f_{p^e}$.
\end{thm}

\begin{proof}
Since $l_1=l_2=l_3=l_4=0$, we have 
$$h(x)=(2-k)x^2+2(6-k)x+4k.$$
If $k\neq 6$, the proof follows from the fact that $$h(0)=4k=h\Big(\frac{2(k-6)}{(2-k)}\Big).$$
If $k=6$, the proof follows from the fact that $h(x)=-4x^2+24$ is not a PP of $\f_{p^e}$.
\end{proof}




\subsection{The case $k=2$}

\begin{thm}
Let $p\geq 3$ and $k=2$. Assume that 
\begin{enumerate}
\item exactly two of $l_1, l_2, l_3$, and $l_4$ are non-zero, and 
\item $(-\frac{1}{2})$ is a quadratic reside modulo $p$. 
\end{enumerate} 
Then $D_{n,k}(1,x)$ is not a PP of $\f_{p^e}$.
\end{thm}

\begin{proof}
Since exactly two of $l_1, l_2, l_3$, and $l_4$ are non-zero, we have $h(0)=4$. The proof follows from the fact that 
$$h(-1/2)=4=h(0).$$
\end{proof}

\begin{thm}
Let $p\geq 5$ and $k=2$. Assume that 
\begin{enumerate}
\item exactly three of $l_1, l_2, l_3$, and $l_4$ are non-zero, and 
\item $(-\frac{3}{4})$ is a quadratic reside modulo $p$. 
\end{enumerate} 
Then $D_{n,k}(1,x)$ is not a PP of $\f_{p^e}$.
\end{thm}

\begin{proof}
Since exactly three of $l_1, l_2, l_3$, and $l_4$ are non-zero, we have $h(0)=2$. The proof follows from the fact that 
$$h(-3/4)=2=h(0).$$
\end{proof}

\begin{thm}
Let $p\geq 5$ and $k=2$. Assume that 
\begin{enumerate}
\item $l_1, l_2, l_3$, and $l_4$ are all non-zero, and 
\item $p\equiv 1\pmod{4}$.
\end{enumerate} 
Then $D_{n,k}(1,x)$ is not a PP of $\f_{p^e}$.
\end{thm}

\begin{proof}
Since $l_1, l_2, l_3$, and $l_4$ are all non-zero, we have $h(0)=0$. The proof follows from the fact that $-1$ is a quadratic residue modulo $p$ whenever $p\equiv 1\pmod{4}$, and 
$$h(-1)=0=h(0).$$
\end{proof}

\begin{rmk}
The case in which exactly three of $l_1$, $l_2$, $l_3$, and $l_4$ are nonzero is explained in Section 3. 
\end{rmk}



\begin{rmk}
If $k=2$ and $l_1=l_2=l_3=l_4=0$, then $D_{p^{l_1}+p^{l_2}+p^{l_3}+p^{l_4},k}(1,x)$ is a PP of $\f_{p^e}$.
\end{rmk}


\section{A generalization}\label{S6}

In this Section, we consider the case $n=p^{l_1}+p^{l_2}+\cdots +p^{l_i}$, where $l_1, \ldots, l_i$ are non-negative integers. 

For any $i\in \mathbb{Z}^+$, we have
\begin{equation}\label{E61}
\begin{split}
&(1+u)^{p^{l_1}+p^{l_2}+\cdots +p^{l_i}}\cr
&=u^{p^{l_1}+p^{l_2}+\cdots +p^{l_i}}+\displaystyle\sum_{j_1, j_2, \ldots, j_{i-1}\in \{l_1, l_2, \ldots, l_i\}}\,u^{p^{j_1}+p^{j_2}+ \dots +p^{j_{i-1}}}+\cdots + \displaystyle\sum_{j_1, j_2 \in \{l_1, l_2, \ldots, l_i\}}\,u^{p^{j_1}+p^{j_2}}\cr
&+\displaystyle\sum_{j_1 \in \{l_1, l_2, \ldots, l_i\}}\,u^{p^{j_1}}+1.
\end{split}
\end{equation}

Let $i$ be odd. Then 

\begin{equation}\label{E62}
\begin{split}
&(1-u)^{p^{l_1}+p^{l_2}+\cdots +p^{l_i}}\cr
&=-u^{p^{l_1}+p^{l_2}+\cdots +p^{l_i}}+\displaystyle\sum_{j_1, j_2, \ldots, j_{i-1}\in \{l_1, l_2, \ldots, l_i\}}\,u^{p^{j_1}+p^{j_2}+ \dots +p^{j_{i-1}}}-\cdots + \displaystyle\sum_{j_1, j_2 \in \{l_1, l_2, \ldots, l_i\}}\,u^{p^{j_1}+p^{j_2}}\cr
&-\displaystyle\sum_{j_1 \in \{l_1, l_2, \ldots, l_i\}}\,u^{p^{j_1}}+1.
\end{split}
\end{equation}

Thus, we have the following when $i$ is odd. 

\begin{equation}\label{E63}
\begin{split}
&(1+u)^{p^{l_1}+p^{l_2}+\cdots +p^{l_i}}+(1-u)^{p^{l_1}+p^{l_2}+\cdots +p^{l_i}}\cr
&=2\displaystyle\sum_{j_1, j_2, \ldots, j_{i-1}\in \{l_1, l_2, \ldots, l_i\}}\,u^{p^{j_1}+p^{j_2}+ \dots +p^{j_{i-1}}}+2\displaystyle\sum_{j_1, j_2, \ldots, j_{i-3}\in \{l_1, l_2, \ldots, l_i\}}\,u^{p^{j_1}+p^{j_2}+ \dots +p^{j_{i-3}}}\cr
&+\cdots + 2\displaystyle\sum_{j_1, j_2 \in \{l_1, l_2, \ldots, l_i\}}\,u^{p^{j_1}+p^{j_2}}+2.
\end{split}
\end{equation}

\begin{equation}\label{E64}
\begin{split}
&(1+u)^{p^{l_1}+p^{l_2}+\cdots +p^{l_i}}\,(1-u)-(1+u)(1-u)^{p^{l_1}+p^{l_2}+\cdots +p^{l_i}}\cr
&=2u^{p^{l_1}+p^{l_2}+\cdots +p^{l_i}}-2\displaystyle\sum_{j_1, j_2, \ldots, j_{i-1}\in \{l_1, l_2, \ldots, l_i\}}\,u^{p^{j_1}+p^{j_2}+ \dots +p^{j_{i-1}}+1}\cr
&+2\displaystyle\sum_{j_1, j_2, \ldots, j_{i-2}\in \{l_1, l_2, \ldots, l_i\}}\,u^{p^{j_1}+p^{j_2}+ \dots +p^{j_{i-2}}}-2\displaystyle\sum_{j_1, j_2, \ldots, j_{i-3}\in \{l_1, l_2, \ldots, l_i\}}\,u^{p^{j_1}+p^{j_2}+ \dots +p^{j_{i-3}}+1}\cr
&+\cdots - 2\displaystyle\sum_{j_1, j_2 \in \{l_1, l_2, \ldots, l_i\}}\,u^{p^{j_1}+p^{j_2}+1}+2\displaystyle\sum_{j_1 \in \{l_1, l_2, \ldots, l_i\}}\,u^{p^{j_1}}-2u.
\end{split}
\end{equation}

Now let $i$ be even. Then 

\begin{equation}\label{E65}
\begin{split}
&(1-u)^{p^{l_1}+p^{l_2}+\cdots +p^{l_i}}\cr
&=u^{p^{l_1}+p^{l_2}+\cdots +p^{l_i}}-\displaystyle\sum_{j_1, j_2, \ldots, j_{i-1}\in \{l_1, l_2, \ldots, l_i\}}\,u^{p^{j_1}+p^{j_2}+ \dots +p^{j_{i-1}}}+\cdots + \displaystyle\sum_{j_1, j_2 \in \{l_1, l_2, \ldots, l_i\}}\,u^{p^{j_1}+p^{j_2}}\cr
&-\displaystyle\sum_{j_1 \in \{l_1, l_2, \ldots, l_i\}}\,u^{p^{j_1}}+1.
\end{split}
\end{equation}

Thus, we have the following when $i$ is even. 

\begin{equation}\label{E66}
\begin{split}
&(1+u)^{p^{l_1}+p^{l_2}+\cdots +p^{l_i}}+(1-u)^{p^{l_1}+p^{l_2}+\cdots +p^{l_i}}\cr
&=2u^{p^{l_1}+p^{l_2}+\cdots +p^{l_i}}+2\displaystyle\sum_{j_1, j_2, \ldots, j_{i-2}\in \{l_1, l_2, \ldots, l_i\}}\,u^{p^{j_1}+p^{j_2}+ \dots +p^{j_{i-2}}}\cr
&+\cdots + 2\displaystyle\sum_{j_1, j_2 \in \{l_1, l_2, \ldots, l_i\}}\,u^{p^{j_1}+p^{j_2}}+2.
\end{split}
\end{equation}

\begin{equation}\label{E67}
\begin{split}
&(1+u)^{p^{l_1}+p^{l_2}+\cdots +p^{l_i}}\,(1-u)-(1+u)(1-u)^{p^{l_1}+p^{l_2}+\cdots +p^{l_i}}\cr
&=-2u^{p^{l_1}+p^{l_2}+\cdots +p^{l_i}+1}+2\displaystyle\sum_{j_1, j_2, \ldots, j_{i-1}\in \{l_1, l_2, \ldots, l_i\}}\,u^{p^{j_1}+p^{j_2}+ \dots +p^{j_{i-1}}}\cr
&-2\displaystyle\sum_{j_1, j_2, \ldots, j_{i-2}\in \{l_1, l_2, \ldots, l_i\}}\,u^{p^{j_1}+p^{j_2}+ \dots +p^{j_{i-2}}+1}+2\displaystyle\sum_{j_1, j_2, \ldots, j_{i-3}\in \{l_1, l_2, \ldots, l_i\}}\,u^{p^{j_1}+p^{j_2}+ \dots +p^{j_{i-3}}}\cr
&+\cdots - 2\displaystyle\sum_{j_1, j_2 \in \{l_1, l_2, \ldots, l_i\}}\,u^{p^{j_1}+p^{j_2}+1}+2\displaystyle\sum_{j_1 \in \{l_1, l_2, \ldots, l_i\}}\,u^{p^{j_1}}-2u.
\end{split}
\end{equation}

Let $n=p^{l_1}+p^{l_2}+\cdots +p^{l_i}$ and $u=2y-1$, where $y\neq \frac{1}{2}$. Then from \eqref{E2.3} we have 

\begin{equation}\label{E700}
\begin{split}
&D_{n,k}(1,y(1-y))\cr
&=k\,\Big[ \displaystyle\frac{\Big(\frac{u+1}{2}\Big)^{p^{l_1}+p^{l_2}+\cdots +p^{l_i}}\Big(\frac{1-u}{2}\Big)-\Big(\frac{u+1}{2}\Big)\Big(\frac{1-u}{2}\Big)^{p^{l_1}+p^{l_2}+\cdots +p^{l_i}}}{u}\Big]\cr
&+\Big(\frac{u+1}{2}\Big)^{p^{l_1}+p^{l_2}+\cdots +p^{l_i}}+\Big(\frac{1-u}{2}\Big)^{p^{l_1}+p^{l_2}+\cdots +p^{l_i}}\cr
\end{split}
\end{equation}

which can be written as

\begin{equation}\label{E711}
\begin{split}
D_{n,k}(1,y(1-y))&=\frac{k}{2^{i+1}}\,\Big[ \displaystyle\frac{(u+1)^{p^{l_1}+p^{l_2}+\cdots +p^{l_i}}(1-u)-(u+1)(1-u)^{p^{l_1}+p^{l_2}+\cdots +p^{l_i}}}{u}\Big]\cr
&+\frac{1}{2^i}\,\Big[(u+1)^{p^{l_1}+p^{l_2}+\cdots +p^{l_i}}+(1-u)^{p^{l_1}+p^{l_2}+\cdots +p^{l_i}}\Big].
\end{split}
\end{equation}

When $y=\frac{1}{2}$, from \eqref{E2.6} we have

\begin{equation}\label{E722}
D_{n,k}\Big(1,\frac{1}{4}\Big)\,=\,\frac{k(p^{l_1}+p^{l_2}+\cdots +p^{l_i}-1)+2}{2^{p^{l_1}+p^{l_2}+\cdots +p^{l_i}}}=\frac{2-k}{2^i}
\end{equation}

for any $i$. Then we have the following:

\vskip 0.1in

\noindent \textbf{Case 1.} Let $i$ be odd and $n=p^{l_1}+p^{l_2}+ \dots +p^{l_i}$. Since $u^2=1-4x$, it follows from \eqref{E63}, \eqref{E64}, \eqref{E711} and \eqref{E722} that for all $x\in \f_q$, we have 

\begin{equation}
\begin{split}
&D_{n,k}(1,x)\cr
&=\frac{k}{2^i}\,(1-4x)^\frac{p^{l_1}+p^{l_2}+ \dots +p^{l_i}-1}{2}+\frac{(2-k)}{2^i}\,\displaystyle\sum_{j_1, j_2, \ldots, j_{i-1}\in \{l_1, l_2, \ldots, l_i\}}\,(1-4x)^\frac{p^{j_1}+p^{j_2}+ \dots +p^{j_{i-1}}}{2}\cr
&+\frac{k}{2^i}\,\displaystyle\sum_{j_1, j_2, \ldots, j_{i-2}\in \{l_1, l_2, \ldots, l_i\}}\,(1-4x)^\frac{p^{j_1}+p^{j_2}+ \dots +p^{j_{i-2}}-1}{2}\cr
&+\frac{(2-k)}{2^i}\,\displaystyle\sum_{j_1, j_2, \ldots, j_{i-3}\in \{l_1, l_2, \ldots, l_i\}}\,(1-4x)^\frac{p^{j_1}+p^{j_2}+ \dots +p^{j_{i-3}}}{2}+\ldots\ldots\cr
&+\frac{(2-k)}{2^i}\,\displaystyle\sum_{j_1, j_2 \in \{l_1, l_2, \ldots, l_i\}}\,(1-4x)^\frac{p^{j_1}+p^{j_2}}{2}+\frac{k}{2^i}\,\displaystyle\sum_{j_1 \in \{l_1, l_2, \ldots, l_i\}}\,(1-4x)^\frac{p^{j_1}-1}{2}+\frac{(2-k)}{2^i}.
\end{split}
\end{equation}

\vskip 0.1in

\noindent \textbf{Case 2.} Let $i$ be even and $n=p^{l_1}+p^{l_2}+ \dots +p^{l_i}$. Since $u^2=1-4x$, it follows from \eqref{E66}, \eqref{E67}, \eqref{E711} and \eqref{E722} that for all $x\in \f_q$, we have 

\begin{equation}\label{even} 
\begin{split}
&D_{n,k}(1,x)\cr
&=\frac{(2-k)}{2^i}\,(1-4x)^\frac{p^{l_1}+p^{l_2}+ \dots +p^{l_i}}{2}+\frac{k}{2^i}\,\displaystyle\sum_{j_1, j_2, \ldots, j_{i-1}\in \{l_1, l_2, \ldots, l_i\}}\,(1-4x)^\frac{p^{j_1}+p^{j_2}+ \dots +p^{j_{i-1}}-1}{2}\cr
&+\frac{(2-k)}{2^i}\,\displaystyle\sum_{j_1, j_2, \ldots, j_{i-2}\in \{l_1, l_2, \ldots, l_i\}}\,(1-4x)^\frac{p^{j_1}+p^{j_2}+ \dots +p^{j_{i-2}}}{2}\cr
&+\frac{k}{2^i}\,\displaystyle\sum_{j_1, j_2, \ldots, j_{i-3}\in \{l_1, l_2, \ldots, l_i\}}\,(1-4x)^\frac{p^{j_1}+p^{j_2}+ \dots +p^{j_{i-3}}-1}{2}+\ldots\ldots\cr
&+\frac{(2-k)}{2^i}\,\displaystyle\sum_{j_1, j_2 \in \{l_1, l_2, \ldots, l_i\}}\,(1-4x)^\frac{p^{j_1}+p^{j_2}}{2}+\frac{k}{2^i}\,\displaystyle\sum_{j_1 \in \{l_1, l_2, \ldots, l_i\}}\,(1-4x)^\frac{p^{j_1}-1}{2}+\frac{(2-k)}{2^i}.
\end{split}
\end{equation}

\subsection{Permutation behaviour of $D_{p^{l_1}+p^{l_2}+\cdots +p^{l_i},k}$}  \hspace*{\fill} \\

We consider the odd $i$ case and even $i$ case separately. 

\vskip 0.1in

\noindent \textbf{Case 1.} Let $i$ be odd and $n=p^{l_1}+p^{l_2}+ \dots +p^{l_i}$. Then

$D_{p^{l_1}+p^{l_2}+ \dots +p^{l_i},k}(1,x)$ is a PP of $\f_{p^e}$ if and only if $h(x)$ is a PP of $\f_{p^e}$, where 

\begin{equation}\label{Oct291} 
\begin{split}
&h(x)=\frac{k}{2^i}\,x^\frac{p^{l_1}+p^{l_2}+ \dots +p^{l_i}-1}{2}+\frac{(2-k)}{2^i}\,\displaystyle\sum_{j_1, j_2, \ldots, j_{i-1}\in \{l_1, l_2, \ldots, l_i\}}\,x^\frac{p^{j_1}+p^{j_2}+ \dots +p^{j_{i-1}}}{2}\cr
&+\frac{k}{2^i}\,\displaystyle\sum_{j_1, j_2, \ldots, j_{i-2}\in \{l_1, l_2, \ldots, l_i\}}\,x^\frac{p^{j_1}+p^{j_2}+ \dots +p^{j_{i-2}}-1}{2}\cr
&+\frac{(2-k)}{2^i}\,\displaystyle\sum_{j_1, j_2, \ldots, j_{i-3}\in \{l_1, l_2, \ldots, l_i\}}\,x^\frac{p^{j_1}+p^{j_2}+ \dots +p^{j_{i-3}}}{2}+\ldots\ldots\cr
&+\frac{(2-k)}{2^i}\,\displaystyle\sum_{j_1, j_2 \in \{l_1, l_2, \ldots, l_i\}}\,x^\frac{p^{j_1}+p^{j_2}}{2}+\frac{k}{2^i}\,\displaystyle\sum_{j_1 \in \{l_1, l_2, \ldots, l_i\}}\,x^\frac{p^{j_1}-1}{2}.
\end{split}
\end{equation}

\noindent \textbf{Case 2.}  Let $i$ be even and $n=p^{l_1}+p^{l_2}+ \dots +p^{l_i}$. Then

$D_{p^{l_1}+p^{l_2}+ \dots +p^{l_i},k}(1,x)$ is a PP of $\f_{p^e}$ if and only if $g(x)$ is a PP of $\f_{p^e}$, where 

\begin{equation}\label{Oct292} 
\begin{split}
g(x)&=\frac{(2-k)}{2^i}\,x^\frac{p^{l_1}+p^{l_2}+ \dots +p^{l_i}}{2}+\frac{k}{2^i}\,\displaystyle\sum_{j_1, j_2, \ldots, j_{i-1}\in \{l_1, l_2, \ldots, l_i\}}\,x^\frac{p^{j_1}+p^{j_2}+ \dots +p^{j_{i-1}}-1}{2}\cr
&+\frac{(2-k)}{2^i}\,\displaystyle\sum_{j_1, j_2, \ldots, j_{i-2}\in \{l_1, l_2, \ldots, l_i\}}\,x^\frac{p^{j_1}+p^{j_2}+ \dots +p^{j_{i-2}}}{2}\cr
&+\frac{k}{2^i}\,\displaystyle\sum_{j_1, j_2, \ldots, j_{i-3}\in \{l_1, l_2, \ldots, l_i\}}\,x^\frac{p^{j_1}+p^{j_2}+ \dots +p^{j_{i-3}}-1}{2}+\ldots\ldots\cr
&+\frac{(2-k)}{2^i}\,\displaystyle\sum_{j_1, j_2 \in \{l_1, l_2, \ldots, l_i\}}\,x^\frac{p^{j_1}+p^{j_2}}{2}+\frac{k}{2^i}\,\displaystyle\sum_{j_1 \in \{l_1, l_2, \ldots, l_i\}}\,x^\frac{p^{j_1}-1}{2}.
\end{split}
\end{equation}

\subsection{Permutation behaviour of $D_{p^{l_1}+p^{l_2},k}$}  \hspace*{\fill} \\

Let $i=2$ in \eqref{even}. Then

$$D_{p^{l_1}+p^{l_2},k}(1,x)=\frac{(2-k)}{4}\,(1-4x)^\frac{p^{l_1}+p^{l_2}}{2}+\frac{k}{4}\,(1-4x)^\frac{p^{l_1}-1}{2}+\frac{k}{4}\,(1-4x)^\frac{p^{l_2}-1}{2}+\frac{(2-k)}{4}.$$

Thus, we have the following results. 

\begin{cor}
Let $k=0$. Then $D_{p^{l_1}+p^{l_2},k}(1,x)$ is a PP of $\f_{p^e}$ if and only if \\ $\textnormal{gcd}(\frac{p^{l_1}+p^{l_2}}{2},p^e-1)=1$. 
\end{cor}

\begin{cor}
Let $p=3$ and $k=2$. Assume that both $l_1$ and $l_2$ are odd. Then $D_{p^{l_1}+p^{l_2},k}(1,x)$ is a PP of $\f_{p^e}$ if and only if the binomial $x^\frac{p^{l_1}-1}{2}+x^\frac{p^{l_2}-1}{2}$ is a PP of $\f_q$. 
\end{cor}

\begin{rmk}
Let $g(x)=x^\frac{p^{l_1}-1}{2}+x^\frac{p^{l_2}-1}{2}$, $p=3$ and $k=2$. When both $l_1$ and $l_2$ are even, $D_{p^{l_1}+p^{l_2},k}(1,x)$ is not a PP of $\f_{p^e}$ since $g(1)=2=g(-1)$. When $l_1$ and $l_2$ have different parity, $D_{p^{l_1}+p^{l_2},k}(1,x)$ is not a PP of $\f_{p^e}$ since $g(0)=0=g(-1)$. 
\end{rmk}

\begin{thm}
Let $p>3$ and $k=2$. Then $D_{p^{l_1}+p^{l_2},k}(1,x)$ is not a PP of $\f_{p^e}$. 
\end{thm}

\begin{proof}
Assume that $p>3$ and $k=2$. Then $D_{p^{l_1}+p^{l_2},k}(1,x)$ is a PP of $\f_{p^e}$ if and only if $f(x)=x^\frac{p^{l_1}-1}{2}+x^\frac{p^{l_2}-1}{2}$ is a PP of $\f_{p^e}$. Let $a$ be a quadratic residue of $p$. Then $a=\beta^2$ for some $1\leq \beta \leq p-1$. We have $f(a)=2=f(1)$. Thus $D_{p^{l_1}+p^{l_2},k}(1,x)$ is not a PP of $\f_{p^e}$. 
\end{proof}

\begin{rmk}
Let $k\neq 0, 2$ and $p>3$. If $l_1=l_2=0$, then $D_{p^{l_1}+p^{l_2},k}(1,x)$ is clearly a PP of $\f_{p^e}$ since every linear polynomial is a PP. 
\end{rmk}

\begin{thm}
Let $k\neq 0, 2$ and $p>3$. Assume that $\frac{2k}{(k-2)}$ is a quadratic residue of $p$. If $l_1\neq 0$ and $\l_2\neq 0$, then $D_{p^{l_1}+p^{l_2},k}(1,x)$ is not a PP of $\f_{p^e}$
\end{thm}

\begin{proof}
Let $k\neq 0, 2$ and $p>3$. Then $D_{p^{l_1}+p^{l_2},k}(1,x)$ is a PP of $\f_{p^e}$ if and only if 
$$f(x)=(2-k)\,x^\frac{p^{l_1}+p^{l_2}}{2}+k\,x^\frac{p^{l_1}-1}{2}+k\,x^\frac{p^{l_2}-1}{2}$$
is a PP of $\f_{p^e}$. Since $l_1$ and $l_2$ are not both zero, we have $f(0)=0$. The proof follows from the fact that $$f\Big(\frac{2k}{(k-2)}\Big)=0=f(0).$$
\end{proof}





\section{More results}

\subsection{The polynomial $D_{rp^l+s,k}(1,x)$}\label{case3}  \hspace*{\fill} \\

In this section, we consider the case $n=rp^l+s$ where $r\geq 1$ and $l, s\geq 0$ are integers.

\begin{prop}\label{P31}
Suppose that $r+s$ is even.  Then $D_{rp^l+s,k}(1,x)$ permutes $\f_{p^e}$ if and only if
\begin{equation*}
\Small{
\begin{split}
&g(x)=(2-k)x^{\frac{rp^l+s}{2}}+k\left[\binom{r}{r}\binom{s}{s-1}\,x^{\frac{rp^l+(s-2)}{2}}+\binom{r}{r-1}\binom{s}{s}\,x^{\frac{(r-1)p^l+(s-1)}{2}}\right]\cr
&+(2-k)\left[\binom{r}{r}\binom{s}{s-2}\,x^{\frac{rp^l+(s-2)}{2}}+\binom{r}{r-1}\binom{s}{s-1}\,x^{\frac{(r-1)p^l+(s-1)}{2}}+\binom{r}{r-2}\binom{s}{s}\,x^{\frac{(r-2)p^l+s}{2}}\right]\cr
&+k\left[\binom{r}{r}\binom{s}{s-3}\,x^{\frac{rp^l+(s-4)}{2}}+\binom{r}{r-1}\binom{s}{s-2}\,x^{\frac{(r-1)p^l+(s-3)}{2}}+\binom{r}{r-2}\binom{s}{s-1}\,x^{\frac{(r-2)p^l+(s-2)}{2}}\right.\cr
&\left. +\binom{r}{r-3}\binom{s}{s}\,x^{\frac{(r-3)p^l+(s-1)}{2}}\right]+\ldots\ldots\cr
&+(2-k)\left[\binom{r}{1}\binom{s}{1}x^{\frac{p^l+1}{2}}+\binom{r}{2}\binom{s}{0}x^{p^l}+\binom{r}{0}\binom{s}{2}x\right]+k\left[\binom{r}{1}\binom{s}{0}x^{\frac{p^l-1}{2}}+\binom{r}{0}\binom{s}{1}\right]+(2-k)
\end{split}
}
\end{equation*}
permutes $\f_{p^e}$.
\end{prop}

\begin{proof}
Let $n=p^{l_1}+p^{l_2}+\cdots +p^{l_i}$ with  $i=r+s$ in \eqref{Oct292}. Then the proof follows by letting $l_1=\cdots =l_r=l$ and $l_{r+1}=\cdots =l_{r+s}=0$ and by counting the number of choices for $j_i$'s.
\end{proof}

\begin{prop}\label{P32}
Suppose that $r+s$ is odd.  Then $D_{rp^l+s,k}(1,x)$ permutes $\f_{p^e}$ if and only if
\begin{equation*}
\Small{
\begin{split}
&h(x)
=kx^{\frac{rp^l+(s-1)}{2}}+(2-k)\left[\binom{r}{r}\binom{s}{s-1}\,x^{\frac{rp^l+(s-1)}{2}}+\binom{r}{r-1}\binom{s}{s}\,x^{\frac{(r-1)p^l+s}{2}}\right]\cr
&+k\left[\binom{r}{r}\binom{s}{s-2}\,x^{\frac{rp^l+(s-3)}{2}}+\binom{r}{r-1}\binom{s}{s-1}\,x^{\frac{(r-1)p^l+(s-2)}{2}}+\binom{r}{r-2}\binom{s}{s}\,x^{\frac{(r-2)p^l+(s-1)}{2}}\right]\cr
&+(2-k)\left[\binom{r}{r}\binom{s}{s-3}\,x^{\frac{rp^l+(s-3)}{2}}+\binom{r}{r-1}\binom{s}{s-2}\,x^{\frac{(r-1)p^l+(s-2)}{2}}+\binom{r}{r-2}\binom{s}{s-1}\,x^{\frac{(r-2)p^l+(s-1)}{2}}\right.\cr
&\left.+\binom{r}{r-3}\binom{s}{s}\,x^{\frac{(r-3)p^l+s}{2}}\right]+\ldots\ldots\cr
&+(2-k)\left[\binom{r}{1}\binom{s}{1}x^{\frac{p^l+1}{2}}+\binom{r}{2}\binom{s}{0}x^{p^l}+\binom{r}{0}\binom{s}{2}x\right]+k\left[\binom{r}{1}\binom{s}{0}x^{\frac{p^l-1}{2}}+\binom{r}{0}\binom{s}{1}\right]+(2-k)
\end{split}
}
\end{equation*}
permutes $\f_{p^e}$.
\end{prop}

\begin{proof}
Let $n=p^{l_1}+p^{l_2}+\cdots +p^{l_i}$ with  $i=r+s$ in \eqref{Oct291} . Then the proof follows by letting $l_1=\cdots =l_r=l$ and $l_{r+1}=\cdots =l_{r+s}=0$ and by counting the number of choices for $j_i$'s.
\end{proof}

\subsection{The polynomial $D_{p^{l}+s,k}(1,x)$}\label{case1}  \hspace*{\fill} \\

In this subsection, we consider the case when $n=p^l+s$ where $l, s\geq 0$. We also explain a few cases in which $D_{p^{l}+s,k}(1,x)$ is not a permutation polynomial of $\f_{p^e}$. 

\begin{prop}\label{meven}
Let $s$ be even. Then $D_{p^{l}+s,k}(1,x)$ permutes $\f_{p^e}$ if and only if
\begin{equation}\label{NF1}
\begin{split}
f(x)=&\displaystyle\sum_{i=0}^{\frac{s}{2}}\left[\binom{s}{s-2i}\,k+\,\binom{s}{s-(2i+1)}\,(2-k)\right]\,x^\frac{p^l+s-2i-1}{2}\cr
&+\displaystyle\sum_{i=0}^{\frac{s}{2}}\left[\binom{s}{s-2i+1}\,k+\,\binom{s}{s-2i}\,(2-k)\right]\,x^\frac{s-2i}{2}
\end{split}
\end{equation}
permutes $\f_{p^e}$.
\end{prop}

\begin{proof} Let $r=1$ in Proposition~\ref{P32}.
\end{proof}

\begin{cor}
Let $s$ be even with $2^{s+1}\equiv s+1\pmod{p}$. Then $D_{p^{l}+s,1}(1,x)$ is not a permutation polynomial of $\f_{p^e}$.
\end{cor}

\begin{proof}
Let $k=1$. The proof follows from the fact that $f(0)=s+1$, $f(1)=2^{s+1}$, and $2^{s+1}\equiv s+1\pmod{p}$. 
\end{proof}

\begin{cor}
Let $s$ be even with $2^{s}\equiv s\pmod{p}$. Then $D_{p^{l}+s,2}(1,x)$ is not a permutation polynomial of $\f_{p^e}$.
\end{cor}

\begin{proof}
Let $k=2$. The proof follows from the fact that $f(0)=2s$, $f(1)=2^{s+1}$, and $2^{s}\equiv s\pmod{p}$. 
\end{proof}

\begin{cor}
Let $s$ be even with $s\equiv 0\pmod{p-1}$. Then $D_{p^{l}+s,0}(1,x)$ is not a permutation polynomial  of $\f_{p^e}$.
\end{cor}

\begin{proof}
Let $k=0$. The proof follows from the fact that $f(0)=2$, $f(1)=2^{s+1}$, and $s\equiv 0\pmod{p-1}$. 
\end{proof}

\begin{prop}\label{modd}
Let $s$ be odd. Then $D_{p^{l}+s,k}(1,x)$ permutes  $\f_{p^e}$ if and only if
\begin{equation}\label{NF2}
\begin{split}
h(x)=&\displaystyle\sum_{i=0}^{\frac{s+1}{2}}\left[\binom{s}{s-2i+1}\,k+\,\binom{s}{s-2i}\,(2-k)\right]\,x^\frac{p^l+s-2i}{2}\cr
&+\displaystyle\sum_{i=0}^{\frac{s+1}{2}}\left[\binom{s}{s-2i}\,k+\,\binom{s}{s-(2i+1)}\,(2-k)\right]\,x^\frac{s-2i-1}{2}
\end{split}
\end{equation}
permutes $\f_{p^e}$.
\end{prop}

\begin{proof} Let $r=1$ in Proposition~\ref{P31}. 
\end{proof}

\begin{cor}
Let $s$ be odd with $2^{s+1}\equiv s+1\pmod{p}$. Then $D_{p^{l}+s,1}(1,x)$ is not a permutation polynomial of $\f_{p^e}$.
\end{cor}

\begin{proof}
Let $k=1$. The proof follows from the fact that $h(0)=s+1$, $h(1)=2^{s+1}$, and $2^{s+1}\equiv s+1\pmod{p}$. 
\end{proof}

\begin{cor}
Let $s$ be odd with $2^{s}\equiv s\pmod{p}$. Then $D_{p^{l}+s,2}(1,x)$ is not a permutation polynomial of $\f_{p^e}$.
\end{cor}

\begin{proof}
Let $k=2$. The proof follows from the fact that $h(0)=2s$, $h(1)=2^{s+1}$, and $2^{s}\equiv s\pmod{p}$. 
\end{proof}

\subsection{The polynomial $D_{rp^{l},k}(1,x)$}\label{case2}  \hspace*{\fill} \\

In this subsection, we consider the case when $n=rp^l$ with $r\geq 1$ and $l\geq 0$. We also explain a few cases in which $D_{rp^l,k}(1,x)$ is not a permutation polynomial of $\f_{p^e}$.

\begin{prop}
Let $r$ be odd. Then $D_{rp^l,k}(1,x)$ permutes $\f_{p^e}$ if and only if
\begin{equation}\label{E5.1}
\begin{split}
f^*(x)= &k\,x^{\frac{rp^l-1}{2}}+(2-k)\,\binom{r}{r-1}\,x^{\frac{(r-1)p^l}{2}}+k\,\binom{r}{r-2}\,x^{\frac{(r-2)p^l-1}{2}}\cr
& +(2-k)\,\binom{r}{r-3}\,x^{\frac{(r-3) p^l}{2}}+\cdots+(2-k)\,\binom{r}{2}\,x^{\frac{2p^l}{2}}+k\,\binom{r}{1}\,x^{\frac{p^l-1}{2}}
\end{split}
\end{equation}
permutes $\f_{p^e}$.
\end{prop}
\begin{proof}
Let $s=0$ in Proposition~\ref{P32}. Then we have
\[
\begin{split}
f(x)=& k\,x^{\frac{rp^l-1}{2}}+(2-k)\,\binom{r}{r-1}\,x^{\frac{(r-1) p^l}{2}}+k\,\binom{r}{r-2}\,x^{\frac{(r-2) p^l-1}{2}}\cr
&+(2-k)\,\binom{r}{r-3}\,x^{\frac{(r-3) p^l}{2}}+\cdots+(2-k)\,\binom{r}{2}\,x^{\frac{2 p^l}{2}}+k\,\binom{r}{1}\,x^{\frac{p^l-1}{2}}+\binom{r}{0}\,(2-k).
\end{split}
\]
\end{proof}

\begin{cor}
Let $r$ be odd with $r\equiv 1\pmod{p-1}$. Then $D_{rp^l,0}(1,x)$ is not a permutation polynomial of $\f_{p^e}$.
\end{cor}

\begin{proof}
Let $k=0$. We note that $f^*(0)=0$ and
\[
\begin{split}
f^*(1)=2\left[\binom{r}{r-1}+\binom{r}{r-3}+\binom{r}{r-5}+\cdots +\binom{r}{2}\right]=2(2^{r-1}-1).
\end{split}
\]
Since $r\equiv 1\pmod{p-1}$, we conclude that $f^*(1)=0$ and thus $D_{rp^l,0}(1,x)$ is not a permutation polynomial of $\f_{p^e}$.
\end{proof}

\begin{prop}
Let $r$ be even. Then $D_{rp^l,k}(1,x)$ permutes $\f_{p^e}$ if and only if
\begin{equation*}\label{E5.2}
\begin{split}
h^*(x)=&\;(2-k)\,x^{\frac{rp^l}{2}}+k\,\binom{r}{r-1}\,x^{\frac{(r-1)p^l-1}{2}}+(2-k)\,\binom{r}{r-2}\,x^{\frac{(r-2)p^l}{2}}\cr
&\;+k\,\binom{r}{r-3}\,x^{\frac{(r-3)p^l-1}{2}}+\cdots+(2-k)\,\binom{r}{2}\,x^{\frac{2p^l}{2}}+k\,\binom{r}{1}\,x^{\frac{p^l-1}{2}}
\end{split}
\end{equation*}
permutes $\f_{p^e}$.
\end{prop}
\begin{proof}

Let $s=0$ in Proposition~\ref{P31}. Then we have
\[
\begin{split}
&h(x)=(2-k)\,x^{\frac{rp^l}{2}}+k\,\binom{r}{r-1}\,x^{\frac{(r-1)p^l-1}{2}}+(2-k)\,\binom{r}{r-2}\,x^{\frac{(r-2)p^l}{2}}\cr
&+k\,\binom{r}{r-3}\,x^{\frac{(r-3)p^l-1}{2}}+\cdots+(2-k)\,\binom{r}{2}\,x^{\frac{2 p^l}{2}}+k\,\binom{r}{1}\,x^{\frac{p^l-1}{2}}+\binom{r}{0}\,(2-k).
\end{split}
\]
\end{proof}

\begin{cor}
Let $r$ be even with $r\equiv 0\pmod{p-1}$. Then $D_{rp^l,1}(1,x)$ is not a permutation polynomial of $\f_{p^e}$.
\end{cor}

\begin{proof}
Let $k=1$. We note that $h^*(0)=0$ and
\[
\begin{split}
h^*(1)=2\left[\binom{r}{r}+\binom{r}{r-1}+\cdots +\binom{r}{1}\right]=2(2^{r}-1).
\end{split}
\]
Since $r\equiv 0\pmod{p-1}$, we have that $h^*(1)=0$ and thus $D_{rp^l,1}(1,x)$ is not a permutation polynomial of $\f_{p^e}$.
\end{proof}


\section{Some generalizations, arithmetic properties and matrix form}

\subsection{The polynomial $D_{rp^{l}-1,k}(1,x)$}\label{more}  \hspace*{\fill} \\

In this subsection, we present a generalization of the following result on reversed Dickson polynomials of the second kind due to Hong, Qin and Zhao.

\begin{prop}\cite[Proposition~2.3]{Hong-Qin-Zhao-FFA-2016-2}\label{HQZ}
Let  $l$ be a positive integer. Then  $D_{p^l-1,1}(1,x)=(1-4x)^{\frac{p^l-1}{2}}$.
\end{prop}

\begin{prop}\label{gen}
Let  $r$ and $l$ be a positive integers with $r\ne 0$. Then
\begin{equation*}\label{TR1}
\begin{split}
D_{rp^l-1,k}(1,x)&=\frac{k}{2^{r-1}}\,\displaystyle\sum_{\mathclap{\substack{ 1\le j\le r \\ j-\textnormal{odd}}}}\,\binom{r}{j}\,(1-4x)^{\frac{j p^l-1}{2}}-\frac{k-1}{2^{r-2}}\,\displaystyle\sum_{\mathclap{\substack{0\leq j\le rp^l-1\\ j-\textnormal{even}}}}\,\binom{rp^l-1}{j}\,(1-4x)^{\frac{j}{2}}.
\end{split}
\end{equation*}
\end{prop}
\begin{proof}
When $x=y(y-1)$ with $y\neq 1/2$ and $n=rp^l-1$ with $r\geq 1$, we obtain from~\eqref{E2.3} that

\begin{equation*}
\begin{split}
D_{rp^l-1,k}(1,x)&=k\,\left[ \displaystyle\frac{y^{rp^l}-(1-y)^{rp^l}}{2y-1}\right]-(k-1)[y^{rp^l-1}+(1-y)^{rp^l-1}].
\end{split}
\end{equation*}

\noindent Let $u=2y-1$. Then

\begin{equation*}\label{NN1}
\begin{split}
D_{rp^l-1,k}(1,x)&=k\,\left[ \displaystyle\frac{\left(\frac{1+u}{2}\right)^{rp^l}-\left(\frac{1-u}{2}\right)^{rp^l}}{u}\right]-(k-1)\left[\left(\frac{1+u}{2}\right)^{rp^l-1}+\left(\frac{1-u}{2}\right)^{rp^l-1}\right]\cr
&=\frac{k}{2^r}\,\left[ \displaystyle\frac{(1+u)^{rp^l}-(1-u)^{rp^l}}{u}\right]-\frac{k-1}{2^{r-1}}\left[(1+u)^{rp^l-1}+(1-u)^{rp^l-1}\right].\cr
\end{split}
\end{equation*}

\noindent For any $r$, we have that

\begin{equation*}\label{NN2}
\begin{split}
(1+u)^{rp^l}-(1-u)^{rp^l}&=(1+u^{p^l})^{r}-(1-u^{p^l})^{r}\cr
&=\displaystyle\sum_{j=0}^{r}\,\binom{r}{j}\,[(u^{p^l})^j-(-u^{p^l})^j] \cr
&=2\,\displaystyle\sum_{\mathclap{\substack{1\le j\le r\\ j-\textnormal{odd}}}}\,\binom{r}{j}\,u^{j p^l}
\end{split}
\end{equation*}

\noindent and
\begin{equation*}\label{NN3}
\begin{split}
(1+u)^{rp^l-1}+(1-u)^{rp^l-1}&=\displaystyle\sum_{j=0}^{rp^l-1}\,\binom{rp^l-1}{j}\,[u^j+(-u)^j] \cr
&=2\,\displaystyle\sum_{\mathclap{\substack{0\leq j\le rp^l-1\\ j-\textnormal{even}}}}\,\binom{rp^l-1}{j}\,u^j.
\end{split}
\end{equation*}

\noindent It follows that

\begin{equation*}\label{NN4}
\begin{split}
D_{rp^l-1,k}(1,x)&=\frac{k}{2^{r-1}}\,\displaystyle\sum_{\mathclap{\substack{1\le j\le r \\ j-\textnormal{odd}}}}\,\binom{r}{j}\,u^{jp^l-1}-\frac{k-1}{2^{r-2}}\,\displaystyle\sum_{\mathclap{\substack{0\leq j\le rp^l-1\\ j-\textnormal{even}}}}\,\binom{rp^l-1}{j}\,u^j\cr
&=\frac{k}{2^{r-1}}\,\displaystyle\sum_{\mathclap{\substack{1\le j\le r\\ j-\textnormal{odd}}}}\,\binom{r}{j}\,(u^2)^{\frac{j p^l-1}{2}}-\frac{k-1}{2^{r-2}}\,\displaystyle\sum_{\mathclap{\substack{0\leq j\le rp^l-1\\ j-\textnormal{even}}}}\,\binom{rp^l-1}{j}\,(u^2)^{\frac{j}{2}}.
\end{split}
\end{equation*}

\noindent Since $u^2=1-4x$,  we have

\begin{equation*}\label{NN5}
\begin{split}
D_{rp^l-1,k}(1,x)&=\frac{k}{2^{r-1}}\,\displaystyle\sum_{\mathclap{\substack{1\le j\le r\\ j-\textnormal{odd}}}}\,\binom{r}{j}\,(1-4x)^{\frac{jp^l-1}{2}}-\frac{k-1}{2^{r-2}}\,\displaystyle\sum_{\mathclap{\substack{0\leq j\le rp^l-1\\ j-\textnormal{even}}}}\,\binom{rp^l-1}{j}\,(1-4x)^{\frac{j}{2}}.
\end{split}
\end{equation*}

\noindent When $y=1/2$, i.e. $x=1/4$, we have

\[
\begin{split}
D_{rp^l-1,k}\left(1,\frac{1}{4}\right)&=\frac{k(rp^l-1-1)+2}{2^{rp^l-1}} =\frac{1-k}{2^{r-2}}\cr
&=\frac{k}{2^{r-1}}\,\displaystyle\sum_{\mathclap{\substack{1\le j\le r\\ j-\textnormal{odd}}}}\,\binom{r}{j}\,(1-4x)^{\frac{j p^l-1}{2}}-\frac{k-1}{2^{r-2}}\,\displaystyle\sum_{\mathclap{\substack{0\leq j\le rp^l-1\\ j-\textnormal{even}}}}\,\binom{rp^l-1}{j}\,(1-4x)^{\frac{j}{2}}.
\end{split}
\]
This ends the proof.
\end{proof}

\begin{rmk}
\noindent Let $r=1$. By Proposition~\ref{gen}, since $\displaystyle\binom{p^l-1}{j}\equiv 1\pmod{p}$ for even $j$, we get
\[
D_{p^l-1,k}(1,x)=k\,(1-4x)^{\frac{p^l-1}{2}}-2(k-1)\,\displaystyle\sum_{j=0}^{\frac{p^l-1}{2}}\,(1-4x)^{j}
\]
\noindent for all $x\in \f_{p^e}$. By considering $k=1$, we immediately obtain Proposition~\ref{HQZ}.
\end{rmk}

\subsection{$D_{n,k}(1, 1-x^2)$}  \hspace*{\fill} \\

In this subsection, we generalize a result in \cite{Hou-Mullen-Sellers-Yucas-FFA-2009}. Let $x_1\neq x_2$ and consider 

\begin{equation}\label{E4.6}
\begin{split}
D_{n,k}(x_1+x_2, x_1x_2)&=k\,\Big[ \displaystyle\frac{x_1^nx_2-x_1x_2^n}{x_1-x_2}\Big]+D_n(x_1+x_2, x_1x_2).
\end{split}
\end{equation}

When $x_1=y$ with $y\neq \frac{1}{2}$ and $x_2=1-x_1$, we have \eqref{E2.3}.

\begin{lem}
Let $l$ be a positive odd integer and let $n=\frac{3^l+1}{2}$. Then in $\f_3[x]$, 
$$D_{n,k}(1, 1-x^2)=\Big(\frac{k}{2}-1\Big)\,D_n(x,1)\,+\,\frac{k}{2}\,\frac{D_{n-1}(x,1)}{x}.$$
\end{lem}

\begin{proof}
Let $x=y^2+y^{-2}$. Since $n$ is even, we have

\[
\begin{split}
D_{n,k}(1, 1-x^2)&= D_{n,k}(1, (2+x)(2-x))\cr
&=k\,\Big[ \displaystyle\frac{(y+y^{-1})^{2n}(2-y^2-\frac{1}{y^2})-(2+y^2+\frac{1}{y^2})(y-y^{-1})^{2n}}{2(y^2+y^{-2})}\Big]-D_n(x,1)\cr
&=k\,\Big[ \displaystyle\frac{(y+y^{-1})^{3^l+1}(2-y^2-\frac{1}{y^2})-(2+y^2+\frac{1}{y^2})(y-y^{-1})^{3^l+1}}{2(y^2+y^{-2})}\Big]-D_n(x,1)\cr
&=k\,\Big[ \displaystyle\frac{2y^{3^l-1}+2y^{-3^l+1}+y^{3^l+3}+y^{-3^l-3}}{2(y^2+y^{-2})}\Big]-D_n(x,1)\cr
&=k\,\Big[ \displaystyle\frac{2(y^2)^\frac{3^l-1}{2}+2(y^2)^\frac{-3^l+1}{2}+(y^2)^\frac{3^l+3}{2}+(y^2)^\frac{-3^l-3}{2}}{2(y^2+y^{-2})}\Big]-D_n(x,1)\cr
&=k\,\Big[ \displaystyle\frac{2(y^2)^{n-1}+2(y^2)^{-(n-1)}+(y^2)^{n+1}+(y^2)^{-(n+1)}}{2(y^2+y^{-2})}\Big]-D_n(x,1)\cr
&=k\,\Big[ \displaystyle\frac{2D_{n-1}(x,1)+D_{n+1}(x,1)}{2\,D_1(x,1)}\Big]-D_n(x,1)\cr
&=k\,\Big[ \displaystyle\frac{xD_{n}(x,1)+D_{n-1}(x,1)}{2x}\Big]-D_n(x,1)\cr
&=\Big(\frac{k}{2}-1\Big)\,D_n(x,1)\,+\,\frac{k}{2}\,\frac{D_{n-1}(x,1)}{x}.
\end{split}
\]

\end{proof}

\begin{rmk}
When $k=0$, we have \cite[Lemma~5.5]{Hou-Mullen-Sellers-Yucas-FFA-2009}.
\end{rmk}

\subsection{$D_{n,k}(1,x)$, $D_{n,0}(1,x)$, $D_{n,1}(1,x)$, and $D_{n,2}(1,x)$}  \hspace*{\fill} \\

In this subsection, we show that $D_{n,k}(1,x)$ can be written in terms of  $D_{n-2,1}(1,x)$ and $D_{n,0}(1,x)$. We also show that $D_{n,k}(1,x)$ can be written in terms of $D_{n-1,2}(1,x)$ and $D_{n,0}(1,x)$. 

Recall 

\begin{equation}\label{E71}
\begin{split}
D_{n,k}(1,y(1-y))&=k\,\Big[ \displaystyle\frac{y^n(1-y)-y(1-y)^n}{2y-1}\Big]+D_n(1,y(1-y)),
\end{split}
\end{equation}

where $y\neq \frac{1}{2}$. 

\eqref{E71} can be rewritten as 

\begin{equation}\label{E72}
\begin{split}
D_{n,k}(1,y(1-y))&=k\,y(1-y)\,\Big[ \displaystyle\frac{y^{n-1}-(1-y)^{n-1}}{2y-1}\Big]+D_n(1,y(1-y)),
\end{split}
\end{equation}

which  gives the following result. 

When $y\neq \frac{1}{2}$,

\[
D_{n,1}(1,y(1-y))=\displaystyle\frac{y^{n+1}-(1-y)^{n+1}}{2y-1}
\]

and

\[
D_{n,2}(1,y(1-y))=\dfrac{y^n-(1-y)^n}{2y-1}
\]

are the functional expressions of the reversed Dickson polynomial of the second kind and third kind, respectively. 

Then from \eqref{E72} we have

\begin{equation}\label{E73}
\begin{split}
D_{n,k}(1,x)&=k\,x\,D_{n-2,1}(1,x)+D_n(1,x), \,\,\,\,n\geq 2
\end{split}
\end{equation}

 and 
 
\begin{equation}\label{E74}
\begin{split}
D_{n,k}(1,x)&=k\,x\,D_{n-1,2}(1,x)+D_n(1,x), \,\,\,\,n\geq 1.
\end{split}
\end{equation}

When $y=\frac{1}{2}$, i.e. $x=\frac{1}{4}$, we have 

\[
D_{n,k}(1,\frac{1}{4})=\displaystyle\frac{k(n-1)+2}{2^n}=\displaystyle{\frac{k}{4}\,\frac{(n-1)}{2^{n-2}}+\frac{2}{2^n}}=\,\displaystyle\frac{k}{4}\,D_{n-2,1}(1,\frac{1}{4})+D_n(1,\frac{1}{4}).
\]

and

\[
D_{n,k}(1,\frac{1}{4})=\displaystyle\frac{k(n-1)+2}{2^n}=\displaystyle{\frac{k}{4}\,\frac{2(n-1)}{2^{n-1}}+\frac{2}{2^n}}=\,\displaystyle\frac{k}{4}\,D_{n-1,2}(1,\frac{1}{4})+D_n(1,\frac{1}{4}).
\]

Thus, for all $x\in \f_q$ we have

\begin{equation}\label{E75}
\begin{split}
D_{n,k}(1,x)&=k\,x\,D_{n-2,1}(1,x)+D_n(1,x),\,\,\,\,n\geq 2
\end{split}
\end{equation}

 and 
 
\begin{equation}\label{E76}
\begin{split}
D_{n,k}(1,x)&=k\,x\,D_{n-1,2}(1,x)+D_n(1,x),\,\,\,\,n\geq 1.
\end{split}
\end{equation}

\subsection{A matrix form of $D_{n,k}(1,x)$}  \hspace*{\fill} \\

In this subsection, we give a matrix representation of $D_{n,k}(1,x)$. A similar matrix representation of the Dickson polynomials of the first kind appeared in \cite[Chapter 2]{Lidl-Mullen-Turnwald-1993}. Using vector notation and the recurrence relation in Proposition~\ref{P2.2} we have

\[
  \begin{matrix}\begin{pmatrix}D_{i+1,k}(1,x), & D_{i+2,k}(1,x) \end{pmatrix}=\end{matrix}
   \begin{pmatrix}D_{i,k}(1,x), & D_{i+1,k}(1,x) \end{pmatrix}
  \begin{pmatrix} 0 & -x \\ 1 & 1 \end{pmatrix} 
\]

\[
  \begin{matrix}\begin{pmatrix}D_{i+2,k}(1,x), & D_{i+3,k}(1,x) \end{pmatrix}=\end{matrix}
   \begin{pmatrix}D_{i,k}(1,x), & D_{i+1,k}(1,x) \end{pmatrix}
  \begin{pmatrix} 0 & -x \\ 1 & 1 \end{pmatrix}^2 
\]

$$\vdots$$

\[
  \begin{matrix}\begin{pmatrix}D_{i+n,k}(1,x), & D_{i+n+1,k}(1,x) \end{pmatrix}=\end{matrix}
   \begin{pmatrix}D_{i,k}(1,x), & D_{i+1,k}(1,x) \end{pmatrix}
  \begin{pmatrix} 0 & -x \\ 1 & 1 \end{pmatrix}^n
\]

Set $i=0$.

\[
  \begin{matrix}\begin{pmatrix}D_{n,k}(1,x), & D_{n+1,k}(1,x) \end{pmatrix}=\end{matrix}
   \begin{pmatrix}D_{0,k}(1,x), & D_{1,k}(1,x) \end{pmatrix}
  \begin{pmatrix} 0 & -x \\ 1 & 1 \end{pmatrix}^n
\]

\[
  \begin{matrix}\begin{pmatrix}D_{n,k}(1,x), & D_{n+1,k}(1,x) \end{pmatrix}=\end{matrix}
   \begin{pmatrix}2-k, & 1\end{pmatrix}
  \begin{pmatrix} 0 & -x \\ 1 & 1 \end{pmatrix}^n
\]

\[
  D_{n,k}(1,x) =
   \begin{pmatrix}2-k, & 1\end{pmatrix}
  \begin{pmatrix} 0 & -x \\ 1 & 1 \end{pmatrix}^n
  \begin{pmatrix} 1 \\ 0 \end{pmatrix}
\]


\section*{Acknowledgements}

The author is heavily indebted to the referees and Michael Zieve for their comments and suggestions which improved the paper significantly. He is grateful to Ariane Masuda for the invaluable discussions and comments. He would also like to thank Anthony Iarrobino for his thoughtful comments on the paper.  


\end{document}